\documentclass[draft]{amsart}
\usepackage{amsmath}
\usepackage{amsthm}
\usepackage{amssymb}
\newtheorem{theorem}{Theorem}[section]
\newtheorem{lemma}[theorem]{Lemma}
\newtheorem{proposition}[theorem]{Proposition}

\theoremstyle{definition}

\theoremstyle{remark}
\newtheorem{remark}[theorem]{Remark}

\numberwithin{equation}{section}
\numberwithin{equation}{section}

\begin{document}

\title[The geometric mean error for in-homogeneous self-similar measures]{Asymptotics of the geometric mean error for in-homogeneous self-similar measures}
\author{Sanguo Zhu, Youming Zhou, Yongjian Sheng}
\address{School of Mathematics and Physics, Jiangsu University of Technology,
Changzhou 213001, China.} \email{sgzhu@jsut.edu.cn}
\thanks{S. Zhu is supported by CSC (No. 201308320049), to whom any correspondence should be addressed.}

\begin{abstract}
Let $(f_i)_{i=1}^N$ be a family of contractive similitudes on $\mathbb{R}^q$ satisfying the open set condition. Let $(p_i)_{i=0}^N$ be a probability vector with $p_i>0$ for all $i=0,1,\ldots,N$. We study the asymptotic geometric mean errors $e_{n,0}(\mu),n\geq 1$, in the quantization for the in-homogeneous self-similar measure $\mu$ associated with the condensation system $((f_i)_{i=1}^N,(p_i)_{i=0}^N,\nu)$.
We focus on the following two independent cases: (I) $\nu$ is a self-similar measure on $\mathbb{R}^q$ associated with $(f_i)_{i=1}^N$; (II) $\nu$ is a self-similar measure associated with another family of contractive similitudes $(g_i)_{i=1}^M$ on $\mathbb{R}^q$ satisfying the open set condition and
$((f_i)_{i=1}^N,(p_i)_{i=0}^N,\nu)$ satisfies a version of in-homogeneous open set condition. We show that, in both cases, the quantization dimension $D_0(\mu)$ of $\mu$ of order zero exists and agrees with that of $\nu$, which is independent of the probability vector $(p_i)_{i=0}^N$. We determine the convergence order of $(e_{n,0}(\mu))_{n=1}^\infty$; namely, for $D_0(\mu)=:d_0$, there exists a constant $D>0$, such that
\[
D^{-1}n^{-\frac{1}{d_0}}\leq e_{n,0}(\mu)\leq D n^{-\frac{1}{d_0}}, n\geq 1.
\]
\end{abstract}
\maketitle

\section{Introduction}
Let $\mathcal{D}_n:=\{\alpha\subset\mathbb{R}^q:1\leq{\rm card}(\alpha)\leq n\}$ for $n\geq 1$. Let $\nu$  be a Borel probability measure on $\mathbb{R}^q$. The $n$th quantization error  for $\nu$ of order $r$ is defined by (see \cite{GL:00,GL:04}):
\begin{eqnarray}\label{quanerrordef}
e_{n,r}(\nu):=\left\{\begin{array}{ll}\inf_{\alpha\in\mathcal{D}_n}\big(\int d(x,\alpha)^{r}d\nu(x)\big)^{\frac{1}{r}}
,\;\;\;\;\;\;r>0,\\
\inf_{\alpha\in\mathcal{D}_n}\exp\int\log d(x,\alpha)d\nu(x),\;\;\;r=0.\end{array}\right.
\end{eqnarray}
Here $d(\cdot,\cdot)$ is the metric induced by an arbitrary norm on $\mathbb{R}^q$.
For $r>0$, $e_{n,r}(\nu)$ agrees with the error in the approximation of $\nu$ by discrete probability measures supported on at most $n$ points,
in the sense of $L_r$-metrics \cite{GL:00}.

The $n$th quantization error of order zero was introduced by Graf and Luschgy and it is also called the $n$th geometric mean
error for $\nu$. By \cite{GL:04}, $e_{n,0}(\nu)$ agrees with the limit of $e_{n,r}(\nu)$ as $r$ tends to zero. In this sense, the quantization with respect to the geometric mean error is a limiting case of that in $L_r$-metrics.

For $s>0$, we define the $s$-dimensional upper and lower quantization coefficient for $\nu$ of order $r$ by (cf. \cite{GL:00,PK:01})
\begin{eqnarray*}
\overline{Q}_r^s(\nu):=\limsup_{n\to\infty}n^{\frac{1}{s}}e_{n,r}(\nu),\;\;
\underline{Q}_r^s(\nu):=\liminf_{n\to\infty}n^{\frac{1}{s}}e_{n,r}(\nu).
\end{eqnarray*}
By \cite{GL:00,PK:01}, the upper (lower) quantization dimension $\overline{D}_r(\nu)$ ($\underline{D}_r(\nu)$) as defined below is exactly the critical point at which the upper (lower) quantization coefficient jumps from zero to infinity:
\begin{eqnarray*}
\overline{D}_r(\nu):=\limsup_{n\to\infty}\frac{\log n}{-\log e_{n,r}(\nu)},\;\underline{D}_r(\nu):=\liminf_{n\to\infty}\frac{\log n}{-\log e_{n,r}(\nu)}.
\end{eqnarray*}
If $\overline{D}_r(\nu)=\underline{D}_r(\nu)$, the common value is denoted by $D_r(\nu)$ and called the quantization dimension for $\nu$. Compared with he upper (lower) quantization, the upper (lower) quantization coefficient provides us with more accurate information on the asymptotics of the geometric mean errors.

The upper and lower quantization dimension of order zero are closely connected with the upper and lower local dimension \cite{Fal:97}:
\[
\underline{\dim}_{\rm loc}\nu(x):=\liminf_{\epsilon\to 0}\frac{\log\nu(B_\epsilon(x))}{\log\epsilon},\;\overline{\dim}_{\rm loc}\nu(x):=\limsup_{\epsilon\to 0}\frac{\log\nu(B_\epsilon(x))}{\log\epsilon}.
\]
Here $B_\epsilon(x)$ denotes the closed ball of radius $\epsilon$ which is centered at a point $x\in\mathbb{R}^q$. In fact, as we showed in \cite{zhu:12}, if the upper and lower local dimension are both equal to $s$ for $\nu$-a.e. $x$, then $D_0(\nu)$ exists and equals $s$. Thus, the geometric mean error $e_{n,0}(\nu)$ connects the local and global behavior of $\nu$ in a natural manner.

Next, let us recall some known results. Let $(f_i)_{i=1}^N$ be a family of contractive similitudes on $\mathbb{R}^q$ with contraction ratios $(s_i)_{i=1}^N$. By \cite{Hut:81}, there exists a unique Borel probability measure $\nu$ satisfying
\[
\nu=q_1\nu\circ f_1^{-1}+q_2\nu\circ f_2^{-1}+\cdots+q_N\nu\circ f_N^{-1}.
\]
This measure is called the self-similar measure associated with $(f_i)_{i=1}^N$ and a probability vector $(q_i)_{i=1}^N$.
We say that $(f_i)_{i=1}^N$ satisfies the open set condition (OSC), if there exists a non-empty bounded open set $U$ such that
$f_i(U)\cap f_j(U)=\emptyset$ for all $1\leq i\neq j\leq N$ and $\bigcup_{i=1}^Nf_i(U)\subset U$.
Let $k_r$ be given by
\[
k_0:=\frac{\sum_{i=1}^Nq_i\log q_i}{\sum_{i=1}^Nq_i\log s_i};\;\;\sum_{i=1}^N(q_is_i^r)^{\frac{k_r}{k_r+r}}=1,\;r>0.
\]
Assume that $(f_i)_{i=1}^N$ satisfies the OSC. Then, Graf and Luschgy proved \cite{GL:01,GL:04}
\begin{eqnarray}\label{gl}
D_r(P)=k_r,\;\;0<\underline{Q}_r^{k_r}(\mu)\leq\overline{Q}_r^{k_r}(\mu)<\infty,\;r\geq 0.
\end{eqnarray}
One may see \cite{GL:00,Kr:08,PK:01} for more related results.

In the present paper, we study the asymptotic geometric mean error for in-homogeneous self-similar measures. We refer to \cite{GL:00,GL:04} for mathematical foundations of quantization theory and \cite{GN:98} for its background in engineering technology. One may see \cite{Las:06,Olsen:08} for recent results on such measures. As above, let $(f_i)_{i=1}^N$ be a family of contractive similitudes.  According to \cite{Hut:81}, there exists a unique non-empty compact set $E$ such that
\begin{eqnarray}\label{s14}
E=f_1(E)\cup\cdots\cup f_N(E).
\end{eqnarray}
This set $E$ is called the self-similar set associated with $(f_i)_{i=1}^N$.
Let $\nu$ be a Borel probability measure on $\mathbb{R}^q$ with compact support  $C$ and $(p_i)_{i=0}^N$ a probability vector with $p_i>0$ for all $0\leq i\leq N$. Then, by \cite{Bar:88,Las:06,Olsen:08}, there exists a unique a Borel probability measure $\mu$ satisfying
\begin{eqnarray}\label{s55}
\mu=p_0\nu+\sum_{i=1}^Np_i\mu\circ f_i^{-1}.
\end{eqnarray}
The support $K$ of $\mu$ is the unique nonempty compact set satisfying (cf. \cite{Las:06,Olsen:08})
\begin{eqnarray}\label{s15}
K=C\cup f_1(K)\cup\cdots\cup f_N(K).
\end{eqnarray}
Following \cite{Olsen:08}, we call $\mu$ the in-homogeneous self-similar measure (ISM) associated with $(f_i)_{i=1}^N,(p_i)_{i=0}^N$ and $\nu$. We focus on the following two disjoint classes of ISMs.

\textbf{Case I}: Assume that $(f_i)_{i=1}^N$ satisfies the OSC; the measure $\nu$ as involved in (\ref{s55}) is a self-similar measure associated with $(f_i)_{i=1}^N$ and a probability vector $(t_i)_{i=1}^N$ with $t_i>0$  for all $0\leq i\leq N$. Note that $C={\rm supp}(\nu)=E$; by (\ref{s14}) and (\ref{s15}), one easily sees that $K=E$. In this case, the support of $\mu$ is a self-similar set; however, its mass distribution is more convoluted than the following Case II (cf. Lemma \ref{g11}).

\textbf{Case II}: The measure $\nu$ as involved in (\ref{s55}) is a self-similar measure associated with $(g_i)_{i=1}^M$ and a probability vector $(t_i)_{i=1}^M$ with $t_i>0$  for all $0\leq i\leq M$, where $(g_i)_{i=1}^M$ is a family of contractive similitudes satisfying the OSC with contraction ratios $(c_i)_{i=1}^M$. Let ${\rm cl}(A),\partial(A)$ and ${\rm int}(A)$ respectively denote the closure, boundary and interior in $\mathbb{R}^q$ of a set $A$. In this case, we always assume the following in-homogeneous open set condition (IOSC) which is a modified version of the IOSC in \cite{Olsen:08}: there exists a bounded non-empty open set $U$ such that

(1) $f_i(U)\subset U$ for all $1\leq i\leq N$;

(2) $f_i(U),1\leq i\leq N$, are pairwise disjoint;

(3) $E\cap U\neq\emptyset$ and $C\subset U$;

(4) $\nu(\partial(U))=0$; $C\cap f_i({\rm cl}(U))=\emptyset$ for all $1\leq i\leq N$.

\begin{remark}
(r1) Compared with the ISMs in Case I, the mass distribution of $\mu$ in Case II is simpler, but its support is much more complicated (cf. Lemma \ref{g112}). In addition, in Case I, we have $C=K=E$, thus, the second part of (4) of the IOSC is violated in an extreme manner.

(r2) In Cases I, II, $(f_i)_{i=1}^N$ satisfies the OSC; in Case II, $(g_i)_{i=1}^M$ satisfies the OSC. Thus, by \cite{GL:01}, in both cases, $D_0(\nu)$ exists and (\ref{gl}) is applicable.

(r3) As we will see, no confusion could arise, although we respectively denote by $(t_i)_{i=1}^N$ and $(t_i)_{i=1}^M$ the probability vectors in Case I and Case II.
\end{remark}
In order to study the asymptotic geometric mean errors for $\mu$, we usually need to consider finite maximal antichains (see section 2 for the definition) of the following form (cf. \cite[p.708]{GL:04}):
\begin{eqnarray*}
\{\sigma\in\Omega^*:\mu(E_{\sigma^-})\geq\epsilon>\mu(E_\sigma)\},\;\epsilon>0.
\end{eqnarray*}
However, for an ISM  in Case I, the mass distribution is rather convoluted and will be very difficult to analyze after taking logarithms. We will choose a suitable sequence of finite maximal antichains according to the mass distribution of $\nu$. Unlike the $L_r$-quantization for $r>0$, where, the quantization coefficient of order $r$ can be infinite, we will prove that the upper quantization coefficient for $\mu$ of order zero is always finite. More exactly,
\begin{theorem}\label{mthm}
Let $\mu$ be an ISM in Case I or Case II. Set
\[
d_0=\frac{\sum_{i=1}^Nt_i\log t_i}{\sum_{i=1}^Nt_i\log s_i}\;\;{\rm for\;Case\; I};\;\;d_0=\frac{\sum_{i=1}^Mt_i\log t_i}{\sum_{i=1}^Mt_i\log c_i}\;\;{\rm for\;Case\; II}.
\]
We have, $0<\underline{Q}_0^{d_0}(\mu)\leq\overline{Q}_0^{d_0}(\mu)<\infty$. In particular, $D_0(\mu)$ exists and equals $d_0$.
\end{theorem}

The remaining part of the paper is organized as follows. In section 2, we collect some basic facts on the ISMs in Cases I, II. The proofs for the dimensional result and the positivity of the lower quantization coefficient will be given in section 3.
In sections 4, 5, we are devoted to the finiteness of the upper quantization coefficient of order zero, respectively for measures $\mu$ in Case I and Case II.

Except for some basic facts about the mass distribution of $\mu$, the subsequent proofs will be given in a self-contained manner. In addition, the OSC and the IOSC are required only to obtain the rules of the mass distributions of ISMs; these conditions will not be explicitly used in section 3-5.

\section{Preliminaries}
First let us recall some notations and definitions. Set
\[
\Omega_n:=\{1,\ldots, N\}^n,\;\Phi_n:=\{1,\ldots, M\}^n, \;\Omega^*:=\bigcup_{n=1}^\infty\Omega_n, \;\Phi^*:=\bigcup_{n=1}^\infty\Phi_n.
\]
 We define $|\sigma|:=n$ for $\sigma\in\Omega_n\cup\Phi_n$ and $\sigma|_0=\theta:=$empty word. For
any $\sigma\in\Omega^*$ with $|\sigma|\geq n$, we
write $\sigma|_n:=(\sigma_1,\ldots,\sigma_n)$.
For $0\leq h<n$ and $\sigma\in\Omega_n$, we set
\begin{eqnarray*}
\sigma^{(l)}_{-h}:=(\sigma_{h+1},\ldots,\sigma_n),\;\;\sigma^-:=\sigma|_{|\sigma|-1}.
\end{eqnarray*}
Clearly, $\sigma^{(l)}_{-0}:=\sigma$. For $\sigma,\tau\in\Omega^*$, we write
\[
\sigma\ast\tau:=(\sigma_1,\ldots,\sigma_{|\sigma|},\tau_1,\ldots,\tau_{|\tau|}).
\]

If $\sigma,\tau\in\Omega^*$ and
$|\sigma|\leq|\tau|,\sigma=\tau|_{|\sigma|}$, then we write $\sigma\preceq\tau$ and call $\sigma$ a predecessor of $\tau$, and $\tau$ a descendant of $\sigma$; if $\sigma\preceq\tau$ and $\sigma\neq\tau$, we write $\sigma\precneqq\tau$ call $\tau$ a \emph{proper descendant} of $\sigma$. Two words $\sigma,\tau\in\Omega^*$ are said to be incomparable if we have neither $\sigma\preceq\tau$
nor $\tau\preceq\sigma$. A finite set $\Gamma\subset\Omega^*$ is
called a finite anti-chain if any two words $\sigma,\tau$ in
$\Gamma$ are incomparable. A finite anti-chain is said to be maximal if
any word $\sigma\in\Omega^{\mathbb{N}}$ has a predecessor in $\Gamma$. Finite maximal anti-chains in $\Phi^*$ and all the above notations for words in $\Phi^*$ are defined in the same way as for words in $\Omega^*$.

Recall that $s_i$ is the contraction ratio of $f_i,1\leq i\leq N$, and $c_i$ is the contraction ratio of $g_i,1\leq i\leq M$. For $\sigma\in\Omega_n$, set and
\begin{eqnarray*}
s_\sigma:=\prod_{h=1}^ns_{\sigma_h},\;p_\sigma:=\prod_{h=1}^np_{\sigma_h};\;f_{\sigma}:=f_{\sigma_1}\circ\cdots\circ f_{\sigma_n},\;E_\sigma:=f_\sigma(E).
\end{eqnarray*}
For every $n\geq 1$ and $\rho\in\Phi_n$, we define
\begin{eqnarray*}
t_\rho:=\prod_{h=1}^nt_{\rho_h},\;\;c_\rho:=\prod_{h=1}^nc_{\rho_h},\;\;g_\rho:=g_{\rho_1}\circ\cdots\circ g_{\rho_n},\;\;C_\rho:=g_\rho(C).
\end{eqnarray*}

With the next two lemmas, we collect some basic facts on the ISMs in Case I and Case II. These facts are easy consequences of the definition of an ISM and the conditions in Case I, II. We refer to \cite{zhu:13,zhu:14} for the proofs.
\begin{lemma}\label{g11}(see \cite[Lemma 2.1]{zhu:13})
Let $\mu$ be an ISM in Case I. we have
\begin{eqnarray}\label{s12}
\mu(E_\sigma)=\sum_{h=0}^{k-1}p_0p_{\sigma|_h} t_{\sigma_{-h}^{(l)}}+p_\sigma,\;\;\sigma\in\Omega_k,\;\;k\geq 1.
\end{eqnarray}
\end{lemma}

To get a description of the support of an ISM in Case II, we write
\begin{eqnarray}\label{s30}
\Gamma(\sigma,h):=\{\tau\in\Omega_{|\sigma|+h}:\sigma\preceq\tau\},\;\;\Gamma^*(\sigma):=\bigcup_{h\geq 1}\Gamma(\sigma,h).
\end{eqnarray}
We see that $\Gamma^*(\sigma)$ is the set of all proper descendants of $\sigma$.

For a finite maximal antichain $\Upsilon\subset\Omega^*$, we define
\[
l(\Upsilon):=\min_{\rho\in\Upsilon}|\rho|,\;\;L(\Upsilon):=\max_{\rho\in\Upsilon}|\rho|.
\]
For each $\sigma\in\Omega_{l(\Upsilon)}$, we define
\begin{eqnarray*}
\Lambda_\Upsilon(\sigma):=\{\tau\in\Omega^*:\sigma\preceq\tau,\Gamma^*(\tau)\cap\Upsilon\neq\emptyset\},\;\;\Lambda_\Upsilon^*:=\bigcup_{\sigma\in\Omega_{l(\Upsilon)}}\Lambda_\Upsilon(\sigma).
\end{eqnarray*}
One can see that $\Lambda_\Upsilon(\sigma)$ consists of all descendants of $\sigma$ which have a proper descendant in $\Upsilon$. For example, if $\sigma\preceq\tau\precneqq\omega$ and $\omega\in\Upsilon$, then $\tau\in\Lambda_\Upsilon(\sigma)$.

\begin{lemma}\label{g112}(see \cite[Lemmas 2.2, 2.3]{zhu:14})
Let $\mu$ be an ISM in Case II. Then

(i) For a finite maximal antichain $\Upsilon$ in $\Omega^*$, we have
\begin{eqnarray}\label{s5}
K=\bigg(\bigcup_{h=0}^{l(\Upsilon)-1}\bigcup_{\sigma\in\Omega_h}f_\sigma(C)\bigg)\cup\bigg(\bigcup_{\sigma\in\Lambda_\Upsilon^*}f_\sigma(C)\bigg)
\cup\bigg(\bigcup_{\sigma\in\Upsilon}f_\sigma(K)\bigg);
\end{eqnarray}

(ii) For every $\sigma\in\Omega^*$ and $\omega\in\Phi^*$, we have
\[
\mu(f_\sigma(K))=p_\sigma,\;\mu(f_\sigma(C_\omega))=p_0p_\sigma t_\omega.
\]
\end{lemma}
Here we remark that, (i) can be easily shown by using (\ref{s15}) and mathematical induction; (ii) is a consequence of (\ref{s15}) the IOSC and some basic results in \cite{Olsen:08}.

Next, we study the $\mu$-measure of a closed ball $B_\epsilon(x)$.
For this, we set
\[
\underline{s}:=\min_{1\leq i\leq N} s_i,\;\;\underline{c}:=\min_{1\leq i\leq M} c_i.
\]
For every $\epsilon\in(0,\underline{s})$, we define
\begin{eqnarray}\label{kk1}
S_\epsilon:=\{\sigma\in\Omega^*:s_{\sigma^-}\geq\epsilon>s_\sigma\},\;\;l(S_\epsilon):=\min_{\sigma\in S_\epsilon}|\sigma|.
\end{eqnarray}
Then for every $h\leq l(S_\epsilon)-1$ and $\sigma\in\Omega_h$, we have, $s_\sigma\geq \epsilon$. Similarly, for $\sigma\in\Lambda_{S_\epsilon}^*$, there exists a $\tau\in S_\epsilon$ with $\sigma\precneqq\tau$; hence, we have, $s_\sigma\geq s_{\tau^-}\geq\epsilon$. We write
\[
\Psi(\epsilon):=\bigg(\bigcup_{h=0}^{l(S_\epsilon)-1}\Omega_h\bigg)\cup\Lambda_{S_\epsilon}^*
\]
Then, for words $\sigma$ in $\Psi(\epsilon)$, we may define
\begin{eqnarray}\label{kk2}
T_\epsilon(\sigma):=\{\rho\in\Phi^*:s_\sigma c_{\rho^-}\geq\epsilon>s_\sigma c_\rho\}.
\end{eqnarray}
Then $T_\epsilon(\sigma)$ is a finite maximal antichain in $\Phi^*$. By the self-similarity of $C$, for each $\sigma\in\Psi(\epsilon)$, we have $C=\bigcup_{\rho\in T_\epsilon(\sigma)}g_\rho(C)$. Thus, by (\ref{s5}), we have
\begin{eqnarray}\label{kk3}
K=\bigg(\bigcup_{\sigma\in\Psi(\epsilon)}\bigcup_{\rho\in T_\epsilon(\sigma)}f_\sigma(g_\rho(C))\bigg)
\cup\bigg(\bigcup_{\sigma\in S_\epsilon}f_\sigma(K)\bigg).
\end{eqnarray}

By Lemma 3.4 of \cite{zhu:14}, there exists an open set $W\supset C$ such that
\begin{eqnarray}\label{gg2}
g_i(W)\subset W,\;\;{\rm cl}(W)\cap f_i({\rm cl}(U))=\emptyset\;\; {\rm for\; all} \;\;1\leq i\leq N.
\end{eqnarray}
Since $C\subset U$, we have $\delta_0:=d(C,U^c)>0$. Thus, by replacing $\epsilon_0$ in \cite[Lemma 3.4]{zhu:14} with $\min\{\epsilon_0,2^{-1}\delta_0\}$, we actually can choose the above $W$ as a subset of $U$.
Since $(g_i)_{i=1}^N$ satisfies the OSC, let $J$ be a nonempty compact set such that (cf. \cite{Gr:95,Schief:94}),
\begin{eqnarray}
&&J={\rm cl}({\rm int}(J)); \;\;{\rm int}(J)\cap C\neq\emptyset;\;\;
g_i(J)\subset J,\;1\leq j\leq M; \\&&g_i({\rm int}(J))\cap g_j({\rm int}(J))=\emptyset,\;\;1\leq i\neq j\leq M.\label{gg1}
\end{eqnarray}
By Lemma 3.3 of \cite{Gr:95}, we have $\nu({\rm int}(J))=1$.
\begin{lemma}
Set $V:={\rm int}(J)\cap W$. Then the following sets are pairwise disjoint:
\[
f_\tau(U),\tau\in S_\epsilon; f_\sigma(g_{\rho}(V)),\rho\in T_\epsilon(\sigma),\sigma\in\Psi(\epsilon).
\]
\end{lemma}
\begin{proof}
(a1) Let $\sigma\in\Psi(\epsilon)$ and $\rho^{(1)},\rho^{(2)}\in T_\epsilon(\sigma)$. Since $T_\epsilon(\sigma)$ is an antichain, $\rho^{(1)},\rho^{(2)}$
are incomparable. Set $h:=\min\{l:\rho^{(1)}_l\neq\rho^{(2)}_l\}$ and write
\[
\rho^{(1)}=\rho^{(1)}|_{h-1}\ast\rho^{(1)}_h\ast\widetilde{\rho^{(1)}},\;\rho^{(2)}=\rho^{(1)}|_{h-1}\ast\rho^{(2)}_h\ast\widetilde{\rho^{(2)}}.
\]
Then, using (\ref{gg1}), we deduce
\begin{eqnarray*}
&&f_\sigma(g_{\rho^{(1)}}(V))\cap f_\sigma(g_{\rho^{(2)}}(V))=f_\sigma(g_{\rho^{(1)}}(V)\cap g_{\rho^{(2)}}(V))\\&&\;\;\subset f_\sigma(g_{\rho^{(1)}}({\rm int}(J))\cap g_{\rho^{(2)}}({\rm int}(J)))\\&&\;\;\subset f_\sigma\circ g_{\rho^{(1)}|_{h-1}}(g_{\rho^{(1)}_h}({\rm int}(J))\cap g_{\rho^{(2)}_h}({\rm int}(J)))=\emptyset.
\end{eqnarray*}

(a2) For distinct words $\sigma,\tau\in\Psi(\epsilon)$ and $\rho^{(1)}\in T_\epsilon(\sigma),\rho^{(2)}\in T_\epsilon(\tau)$, we have
if $\sigma,\tau$ are incomparable, then by (\ref{gg1}) and (2) of the IOSC, we have
\begin{eqnarray*}
&&f_\sigma(g_{\rho^{(1)}}(V))\cap f_\tau(g_{\rho^{(2)}}(V))\subset f_\sigma(g_{\rho^{(1)}}(W))\cap f_\tau(g_{\rho^{(2)}}(W))\\&&\;\;\subset f_\sigma(W)\cap f_\tau(W)\subset f_\sigma(U)\cap f_\tau(U)=\emptyset.
\end{eqnarray*}
if $\sigma,\tau$ are comparable, we may assume that $\sigma\preceq\tau$ and $\tau=\sigma\ast\omega$. Then by (\ref{gg2}),
\begin{eqnarray*}
&&f_\sigma(g_{\rho^{(1)}}(V))\cap f_\tau(g_{\rho^{(2)}}(V))\subset f_\sigma(g_{\rho^{(1)}}(W)\cap f_\omega(g_{\rho^{(2)}}(W))\\&&\;\;\subset f_\sigma(W\cap f_\omega(W))\subset f_\sigma(W\cap f_\omega(U))\subset f_\sigma(W\cap f_{\omega_1}(U))=\emptyset.
\end{eqnarray*}

(a3) Let $\sigma\in\Psi(\epsilon),\rho\in T_\epsilon(\sigma)$ and $\tau\in S_\epsilon$. If $\sigma,\tau$ are incomparable, then by (\ref{gg1}) and (2) of the IOSC, we have
\begin{eqnarray*}
&&f_\sigma(g_{\rho}(V))\cap f_\tau(U)\subset f_\sigma(g_{\rho}(W))\cap f_\tau(U)\\&&\;\;\subset f_\sigma(W)\cap f_\tau(U)\subset f_\sigma(U)\cap f_\tau(U)=\emptyset.
\end{eqnarray*}
If $\sigma,\tau$ are comparable, then $\sigma\preceq\tau$ and $|\sigma|<|\tau|$; in fact, for $\sigma\in\bigcup_{h=0}^{l(S_\epsilon)-1}\Omega_h$, we have, $|\sigma|<|\tau|$; for $\sigma\in\Lambda_S^*$ and there exists a proper descendant $\rho$ of $\sigma$ such that $\rho\in S_\epsilon$, thus it is not possible that $\tau\preceq\sigma$. Therefore, we may write $\tau=\sigma\ast\omega$. Then
\begin{eqnarray*}
&&f_\sigma(g_{\rho}(V))\cap f_\tau(U)\subset f_\sigma(g_{\rho}(W))\cap f_\tau(U)\\&&\;\;\subset f_\sigma(W)\cap f_\tau(U)= f_\sigma(W\cap f_\omega(U))\subset f_\sigma(W\cap f_{\omega_1}(U))=\emptyset.
\end{eqnarray*}

(a4) Let $\sigma,\tau$ be an arbitrary pair of distinct words in $S_\epsilon$. Since $S_\epsilon$ is an antichain, $\sigma,\tau$ are incomparable. Set $h:=\min\{l:\sigma_l\neq\tau_l\}$ and write
\[
\sigma=\sigma|_{h-1}\ast\sigma_h\ast\widetilde{\sigma},\tau=\sigma|_{h-1}\ast\tau_h\ast\widetilde{\tau}.
\]
Then, using the condition (1) and (2) of the IOSC, we deduce
\begin{eqnarray*}
&&f_\sigma(U)\cap f_\tau(U)=f_{\sigma|_{h-1}}(f_{\sigma_h\ast\widetilde{\sigma}}(U)\cap f_{\tau_h\ast\widetilde{\tau}}(U))\subset f_{\sigma|_{h-1}}(f_{\sigma_h}(U)\cap f_{\tau_h}(U))=\emptyset.
\end{eqnarray*}
This completes the proof of the lemma.
\end{proof}
\begin{lemma}\label{pre1}
Let $\mu$ be an ISM in Case I or Case II. Then there exist two constants $\lambda_1,\eta_1>0$ such that
$\sup_{x\in\mathbb{R}^q}\mu(B(x,\epsilon))\leq \lambda_1 \epsilon^\eta_1$ for all $\epsilon>0$.
\end{lemma}
\begin{proof}
In Case I, we have $K=E$, and the OSC is satisfied. One can show the lemma by using the arguments of Graf and Luschgy for self-similar measures (see Proposition 5.1 of \cite{GL:04}).

Next, we assume that $\mu$ is an ISM in Case II.
Note that, $V$ (respectively $U$) contains a ball of some radius $\delta_1>0$ (respectively $\delta_2>0$ ) and is contained a closed ball of radius $|C|$ (respectively $|U|$). Thus, for $\sigma\in\Psi(\epsilon),\rho\in T_\epsilon(\sigma)$, $f_\sigma(g_{\rho}(V))$ contains a ball of radius $\underline{c}\delta_1\epsilon$ and is contained a closed ball of radius $|C|\epsilon$. Similarly, and for every $\tau\in S$, $f_\sigma(U)$ contains a ball of radius $\underline{s}\delta_2\epsilon$ and is contained a closed ball of radius $|U|\epsilon$. By \cite{Hut:81}, there exists an $L_1\geq 1$, which is independent of $\epsilon$, such that $B(x,\epsilon)$ intersects at most $L_1$ of the following sets:
\[
f_\sigma(g_{\rho}({\rm cl}(V))),\;\sigma\in\Psi(\epsilon),\;\rho\in T_\epsilon(\sigma);\;\;f_\tau({\rm cl}(U)),\;\tau\in S_\epsilon.
\]
These sets form a cover of $K$. In fact, we have $\nu({\rm cl}(V))=\nu(V)=1$, so $C={\rm supp}(\nu)\subset{\rm cl}(V)$; in addition, according to \cite{Olsen:08}, we have, $K\subset{\rm cl}(U)$. Thus, by (\ref{kk3}), we obtain
\[
K\subset\bigg(\bigcup_{\sigma\in\Psi(\epsilon)}\bigcup_{\rho\in T_\epsilon(\sigma)}f_\sigma(g_\rho({\rm cl}(V))\bigg)
\cup\bigg(\bigcup_{\sigma\in S_\epsilon}f_\sigma({\rm cl}(U))\bigg).
\]
Set $\delta_3:=\min\{\underline{s},\underline{c}\}$. By (\ref{kk2}), for $\sigma\in\Psi(\epsilon),\rho\in T_\epsilon(\sigma)$, we have
\begin{eqnarray}\label{ss3}
\delta_3^{|\sigma|+|\rho|}<\epsilon,\;\;{\rm implying}\;\;|\sigma|+|\rho|\geq\log\epsilon/\log\delta_3;
\end{eqnarray}
for $\tau\in S$, by (\ref{kk1}), we have
\begin{eqnarray}\label{ss2}
\delta_3^{|\tau|}\leq\underline{s}^{|\tau|}<\epsilon,\;\;{\rm implying}\;\;|\tau|\geq\log\epsilon/\log\delta_3;
\end{eqnarray}
Set $\delta_4:=\max\{\overline{t},\overline{p}\}$. The proof of \cite[Lemma 2.2]{zhu:14} also implies that \[
\mu(f_\sigma(g_{\rho}({\rm cl}(V))))=p_0p_\sigma t_\rho,\;\sigma\in\Omega^*,\;\rho\in\Phi^*.
\]
By (\ref{ss3}) and (\ref{ss2}), we further deduce
\begin{eqnarray*}
\mu(B(x,\epsilon))&\leq& L_1\max\{\sup_{\sigma\in\Psi(\epsilon)}\sup_{\rho\in T_\epsilon(\sigma)}\mu(f_\sigma(g_{\rho}({\rm cl}(V)))),\sup_{\tau\in S}\mu(f_\tau({\rm cl}(U)))\}\\&\leq& L_1\max\{\sup_{\sigma\in\Psi(\epsilon)}\sup_{\rho\in T_\epsilon(\sigma)}p_0p_\sigma t_\rho,\sup_{\tau\in S}p_\sigma\}\\&\leq&L_1\max\{p_0\sup_{\sigma\in\Psi(\epsilon)}\sup_{\rho\in T_\epsilon(\sigma)}\delta_4^{|\sigma|+|\rho|},\sup_{\tau\in S}\delta_4^{|\tau|}\}\\&\leq&
L_1\delta_4^{\log\epsilon/\log\delta_3}=L_1\epsilon^{\log\delta_4/\log\delta_3}.
\end{eqnarray*}
Set $\eta_1:=\log\delta_4/\log\delta_3$. The lemma follows by \cite[Lemma 12.3]{GL:00}.
\end{proof}

\section{Quantization dimension and the lower quantization coefficient}
For a Borel probability measure $\nu$, we simply write $C_n(\nu)$ for $C_{n,0}(\nu)$. Set
\[
\hat{e}_n(\nu):=\log e_{n,0}(\nu)=\inf_{\alpha\in\mathcal{D}_n}\int\log d(x,\alpha)d\nu(x),\;n\geq 1.
\]
\begin{lemma}\label{g2}
Let $\nu$ be a Borel probability measure on $\mathbb{R}^q$ with compact support and $\mu$ be a ISM as defined in (\ref{s5}). Assume that, for some constants $\lambda,\eta>0$,
\begin{eqnarray}\label{s7}
\max\bigg\{\sup_{x\in\mathbb{R}^q}\nu(B(x,\epsilon)),\sup_{x\in\mathbb{R}^q}\mu(B(x,\epsilon))\bigg\}\leq \lambda \epsilon^\eta.
\end{eqnarray}
Then $\underline{D}(\nu)\leq\underline{D}(\mu)\leq\overline{D}(\mu)\leq\overline{D}(\nu)$. In particular, if $D_0(\nu)$ exists and equals $d_0$ then $D_0(\mu)=d_0$. Moreover, if $\underline{Q}_0^{d_0}(\nu)>0$, then we have, $\underline{Q}_0^{d_0}(\mu)>0$.
\end{lemma}
\begin{proof}
Note that $\nu,\mu$ are both compactly supported. By the assumption (\ref{s7}) and Theorem 2.5 of \cite{GL:04}, $C_n(\mu)$ and $C_n(\nu)$ are nonempty for every $n\geq 1$; by Lemma 2.1 of \cite{zhu:12}, $\underline{D}(\nu),\underline{D}(\mu)\geq\eta>0$. Since $f_i,1\leq i\leq N$, are similitudes, we have
\[
\underline{D}(\mu\circ f_i^{-1})=\underline{D}(\mu),\;\;\overline{D}(\mu\circ f_i^{-1})=\overline{D}(\mu),\; 1\leq i\leq N.
\]
Thus, using (\ref{s5}) and \cite[Lemma 2.2]{zhu:12}, we deduce
\begin{eqnarray*}
\frac{\underline{D}(\nu)(\underline{D}(\mu))^N}{p_0(\underline{D}(\mu))^N+(1-p_0)\underline{D}(\nu)(\underline{D}(\mu))^{N-1}}\leq\underline{D}(\mu).
\end{eqnarray*}
Since $\underline{D}(\mu)\geq\eta>0$, it follows that $\underline{D}(\nu)\leq\underline{D}(\mu)$. Analogously, one can see
\begin{eqnarray*}
\frac{\overline{D}(\nu)(\overline{D}(\mu))^N}{p_0(\overline{D}(\mu))^N+(1-p_0)\overline{D}(\nu)(\overline{D}(\mu))^{N-1}}\geq\overline{D}(\mu).
\end{eqnarray*}
This implies that $\overline{D}(\mu)\leq\overline{D}(\nu)$. Hence, if $D_0(\nu)$ exists and equals $d_0$, then we have $D_0(\mu)=d_0$. By (\ref{s5}) and a slight generalization of \cite[Example 4.1]{GL:04},
\begin{eqnarray*}
\hat{e}_n(\mu)&\geq& p_0\hat{e}_n(\nu)+p_1\hat{e}_n(\mu\circ f_1^{-1})+\ldots+p_N\hat{e}_n(\mu\circ f_N^{-1})\\&\geq& p_0\hat{e}_n(\nu)+(1-p_0)\hat{e}_n(\mu)+\sum_{i=1}^Np_i\log c_i.
\end{eqnarray*}
As a consequence, we have, $\hat{e}_n(\mu)\geq\hat{e}_n(\nu)+p_0^{-1}\sum_{i=1}^Np_i\log c_i$. It follows that
$$
\underline{Q}_0^{d_0}(\mu)\geq \exp\bigg(p_0^{-1}\sum_{i=1}^Np_i\log c_i\bigg)\underline{Q}_0^{d_0}(\nu).
$$
This completes the proof of the lemma.
\end{proof}
\begin{remark}{\rm
Using the argument in \cite[Example 4.1]{GL:04}, one can also get
\begin{eqnarray}\label{s31}
\hat{e}_n(\mu)&\leq& p_0\hat{e}_{[n/(N+1)]}(\nu)+p_1\hat{e}_{[n/(N+1)]}(\mu\circ f_1^{-1})+\ldots+p_N\hat{e}_{[n/(N+1)]}(\mu\circ f_N^{-1})\nonumber\\&=& p_0\hat{e}_{[n/(N+1)]}(\nu)+(1-p_0)\hat{e}_{[n/(N+1)]}(\mu)+\sum_{i=1}^Np_i\log c_i.
\end{eqnarray}
For a probability measure $P$, we write $Q_k(P):=d_0^{-1}\log n+\hat{e}_n(\mu), k\geq 1$. By (\ref{s31}),
\[
Q_n(\mu)\leq p_0Q_{[\frac{n}{N+1}]}(\nu))+(1-p_0)Q_{[\frac{n}{N+1}]}(\mu)+\kappa.
\]
where $\kappa:=d_0^{-1}\log (2N+2)+\sum_{i=1}^Np_i\log c_i$.
Thus, by the definition, we obtain
\[
\overline{Q}_0^{d_0}(\mu)\leq e^\kappa(\overline{Q}_0^{d_0}(\nu))^{p_0}(\overline{Q}_0^{d_0}(\mu))^{1-p_0}.
\]
Unfortunately, one can not get the finiteness of $\overline{Q}_0^{d_0}(\mu)$ even if $\overline{Q}_0^{d_0}(\nu)$ is finite.
}\end{remark}
\begin{proposition}\label{g12}
Let $\mu$ be an ISM in Case I or Case II. Then
\[
D_0(\mu)=d_0\;\;{\rm and}\;\;\underline{Q}_0^{d_0}(\mu)>0.
\]
\end{proposition}
\begin{proof}
By (\ref{gl}), $D_0(\nu)=d_0$ and $\underline{Q}_0^{d_0}(\nu)>0$; thus the proposition follows by Lemmas \ref{pre1}, \ref{g2}.
\end{proof}
Let us end this section with the following simple observation:
\begin{lemma}\label{interim}
Let $(n_j)_{j=1}^\infty$ be a increasing sequence of positive integers satisfying $n_j\to\infty\;(j\to\infty)$ and $\sup_{j\geq 1}n_{j+1}/n_j\leq N$ for some constant $N$. Then we have
\[
\overline{Q}_0^{d_0}(\mu)<\infty\Leftrightarrow \overline{P}_0^{d_0}(\mu):=\limsup_{j\to\infty} n_j^{\frac{1}{d_0}}e_{n_j,0}(\mu)<\infty.
\]
\end{lemma}
\begin{proof}
For each $n\geq n_1$, there exists a unique $j\geq 1$, such that $n_j\leq n<n_{j+1}$. Thus, by \cite[Theorem 2.5]{GL:04}, we have
\begin{eqnarray*}
N^{-\frac{1}{d_0}}n_{j+1}^{\frac{1}{d_0}}e_{n_{j+1},0}(\mu)&\leq& n_j^{\frac{1}{d_0}}e_{n_{j+1},0}(\mu)\leq n^{\frac{1}{d_0}}e_{n,0}(\mu)\\&\leq& n_{j+1}^{\frac{1}{d_0}}e_{n_j,0}(\mu)\leq N^{\frac{1}{d_0}}n_j^{\frac{1}{d_0}}e_{n_j,0}(\mu).
\end{eqnarray*}
Hence, we have, $N^{-\frac{1}{d_0}}\overline{P}_0^{d_0}(\mu)\leq\overline{Q}_0^{d_0}(\mu)\leq N^{\frac{1}{d_0}}\overline{P}_0^{d_0}(\mu)$. The lemma follows.
\end{proof}
\section{The upper quantization coefficient of ISMs in Case I}
In this section, $\mu$ always denotes an ISM in Case I. Recall that, in Case I, we have that $K=E$. Let $|A|$ denote the diameter of a set $A$. Without loss of generality, we assume that $|E|=1$. Then $|E_\sigma|=s_\sigma$ for every $\sigma\in\Omega^*$.  Set $\underline{t}:=\min_{1\leq i\leq N}t_i$. The following finite maximal antichains will be adequate for us to study the finiteness of the $d_0$-dimensional upper quantization coefficient:
\begin{eqnarray*}
\Lambda_j:=\{\sigma\in\Omega^*:t_{\sigma^-}\geq \underline{t}^j>t_\sigma\},\;\;j\geq 1.
\end{eqnarray*}
We denote by $\phi_j$ the cardinality of $\Lambda_j$. Define
\[
k_{1j}:=\min_{\sigma\in\Lambda_j}|\sigma|,\;\;k_{2j}:=\max_{\sigma\in\Lambda_j}|\sigma|,\;\;j\geq 1.
\]
\begin{lemma}
(i) For every $j\geq 1$, we have
\begin{eqnarray}\label{s8}
\underline{t}^{-j}\leq\phi_j\leq\underline{t}^{-(j+1)},\;\phi_j\leq\phi_{j+1}\leq \underline{t}^{-2}\phi_j.
\end{eqnarray}
(ii) There exists a constant $B_1$ such that
\begin{eqnarray}\label{s9}
j\leq k_{1j}\leq k_{2j}\leq B_1 j,\;\;j\geq 1.
\end{eqnarray}
\end{lemma}
\begin{proof}
(i) By the definition of $\Lambda_j$ , one can easily see
\begin{eqnarray*}
\phi_j \underline{t}^{j+1}\leq\sum_{\sigma\in\Lambda_j}t_\sigma=\sum_{\sigma\in\Lambda_j}\nu(E_\sigma)=1<\phi_j \underline{t}^j.
\end{eqnarray*}
Thus, (\ref{s8}) holds for all $j\geq 1$.

(ii) Choose arbitrarily $\sigma\in\Lambda_j\cap\Omega_{k_{1j}}$ and $\tau\in\Lambda_j\cap\Omega_{k_{2j}}$. We have
\[
\underline{t}^{k_{1j}}\leq t_\sigma<\underline{t}^j,\;\overline{t}^{k_{1j}}\geq t_\tau\geq\underline{t}^{j+1}.
\]
Hence, (\ref{s9}) is fulfilled for $B_1=2\log\underline{t}/\log\overline{t}$.
\end{proof}

For each $k\geq 1$, we define
\begin{eqnarray*}
d_k:=\frac{\sum_{\sigma\in\Omega_k}\mu(E_\sigma)\log t_\sigma}{\sum_{\sigma\in\Omega_k}\mu(E_\sigma)\log s_\sigma};\;\;\eta_k:=\frac{\sum_{\sigma\in\Lambda_k}\mu(E_\sigma)\log t_\sigma}{\sum_{\sigma\in\Lambda_k}\mu(E_\sigma)\log s_\sigma}.
\end{eqnarray*}
We will frequently use the following equality:
\begin{equation}\label{gg4}
\sum_{\sigma\in\Omega_k}p_\sigma\log t_\sigma=\sum_{h=1}^k\sum_{\sigma\in\Omega_k}\prod_{1\leq l\neq h\leq k}p_{\sigma_l}p_{\sigma_h}\log t_{\sigma_h}=k(1-p_0)^{k-1}\sum_{i=1}^Np_i\log t_i.
\end{equation}
\begin{lemma}\label{g0}
There exists a constant $C_1$ such that $|d_k-d_0|\leq C_1 k^{-1}$ for large $k$.
\end{lemma}
\begin{proof}
Note that $\sum_{\sigma\in\Omega_h}t_\sigma=1$ for $h\geq 1$. Hence, by (\ref{s12}) and (\ref{gg4}),
\begin{eqnarray*}
&&\sum_{\sigma\in\Omega_k}\mu(E_\sigma)\log t_\sigma=\sum_{\sigma\in\Omega_k}\sum_{h=0}^{k-1}p_0p_{\sigma|_h} t_{\sigma_{-h}^{(l)}}\log t_\sigma+\sum_{\sigma\in\Omega_k}p_\sigma\log t_\sigma\\
&&=\sum_{h=0}^{k-1}\sum_{\sigma\in\Omega_k}p_0p_{\sigma|_h} t_{\sigma_{-h}^{(l)}}(\log t_{\sigma|_h}+\log t_{\sigma_{-h}^{(l)}})+\sum_{\sigma\in\Omega_k}p_\sigma\sum_{h=1}^k\log t_{\sigma_h}
\\&&=\sum_{h=0}^{k-1}\sum_{\omega\in\Omega_h}\sum_{\tau\in\Omega_{k-h}} p_0p_\omega t_\tau(\log t_\omega+\log t_\tau)+k(1-p_0)^{k-1}\sum_{i=1}^Np_i\log t_i.
\end{eqnarray*}
Let $u_0:=\sum_{i=1}^Nt_i\log t_i$ and $l_0:=\sum_{i=1}^Nt_i\log s_i$. We further deduce
\begin{eqnarray*}
&&\sum_{\sigma\in\Omega_k}\mu(E_\sigma)\log t_\sigma=p_0\sum_{h=0}^{k-1}h(1-p_0)^{h-1}\sum_{i=1}^Np_i\log t_i\\&&\;\;+p_0\sum_{h=0}^{k-1}(1-p_0)^h(k-h)\sum_{i=1}^Nt_i\log t_i+k(1-p_0)^{k-1}\sum_{i=1}^Np_i\log t_i\\&&=p_0\sum_{h=0}^{k-1}h(1-p_0)^{h-1}\sum_{i=1}^Np_i\log t_i+k(1-(1-p_0)^k)u_0\\&&\;\;-p_0\sum_{h=0}^{k-1}h(1-p_0)^hu_0+k(1-p_0)^{k-1}\sum_{i=1}^Np_i\log t_i.
\end{eqnarray*}
Note that $\sum_{h=0}^\infty h(1-p_0)^{h-1}<\infty$ and $k(1-p_0)^{k-1}\to 0$ as $k\to\infty$. There exists a constant $A_1>0$, such that $|a_k|\leq A_1$, where
\[
a_k:=\sum_{\sigma\in\Omega_k}\mu(E_\sigma)\log t_\sigma-k(1-(1-p_0)^k)u_0.
\]
Analogously, there exists a constant $A_2>0$, such that $|b_k|\leq A_2$ for
\[
b_k:=\sum_{\sigma\in\Omega_k}\mu(E_\sigma)\log s_\sigma-k(1-(1-p_0)^k)l_0.
\]
We denote by $u_k$ and $l_k$ the numerator and denominator of $d_k$. Then
\begin{eqnarray*}
|d_k-d_0|&=&\bigg|\frac{u_k}{l_k}-\frac{u_0}{l_0}\bigg|=\bigg|\frac{k(1-(1-p_0)^k)u_0+a_k}{k
(1-(1-p_0)^k)l_0+b_k}-\frac{u_0}{l_0}\bigg|\\&=&\bigg|\frac{l_0a_k-u_0b_k}{l_0(k
(1-(1-p_0)^k)l_0+b_k)}\bigg|\leq\frac{|u_0+l_0|(A_1+A_2)}{|l_0(p_0kl_0+b_k)|}\\&\leq&\frac{2|u_0+l_0|(A_1+A_2)}{|l_0|^2p_0k}=:\frac{C_1}{k},\;\;{\rm for}\;\;k\geq 2A_2|l_0|^{-2}p_0^{-1};
\end{eqnarray*}
where $C_1:=2|u_0+l_0|(A_1+A_2)|l_0|^{-2}p_0^{-1}$. The lemma follows.
\end{proof}

In the following, we are going to estimate the convergence order of $(\eta_j)_{j=1}^N$. For this purpose, we need to establish a series of lemmas. Write
\begin{equation*}
\mu^{(1)}(\sigma):=\sum_{h=0}^{|\sigma|-1}p_0p_{\sigma|_h} t_{\sigma_{-h}^{(l)}},\;\sigma\in\Omega^*.
\end{equation*}
Then one can see that, for each $1\leq i\leq N$, we have
\begin{equation}\label{s13}
\mu^{(1)}(\sigma\ast i):=\mu^{(1)}(\sigma)t_i+p_0p_\sigma t_i;\;\;\mu(E_\sigma)=\mu^{(1)}(\sigma)+p_\sigma.
\end{equation}
\begin{lemma}\label{g1}
For every $\sigma\in\Omega^*$ and $h\geq 1$, we have
\begin{equation*}
\sum_{\omega\in\Gamma(\sigma,h)}\mu^{(1)}(\omega)=\mu^{(1)}(\sigma)+p_\sigma(1-p_0)(1-(1-p_0)^{h-1})+p_0p_\sigma.
\end{equation*}
\end{lemma}
\begin{proof}
For $h=1$, by (\ref{s13}), we have
\begin{eqnarray*}
\sum_{\omega\in\Gamma(\sigma,1)}\mu^{(1)}(\omega)=\sum_{i=1}^N\mu^{(1)}(\sigma\ast i)=\sum_{i=1}^N(t_i\mu^{(1)}(\sigma)+p_0t_ip_\sigma)=\mu^{(1)}(\sigma)+p_0p_\sigma.
\end{eqnarray*}
Thus, the lemma is true for $h=1$. Next we assume that $h\geq2$. For $l\geq 2$, by (\ref{s13}),
\begin{eqnarray*}
\sum_{\omega\in\Gamma(\sigma,l)}\mu^{(1)}(\omega)&=&\sum_{\tau\in\Omega_l}\mu^{(1)}(\sigma\ast\tau)=\sum_{\omega\in\Omega_{l-1}}\sum_{i=1}^N\mu^{(1)}(\sigma\ast\omega\ast i)\\&=&\sum_{\omega\in\Omega_{l-1}}\sum_{i=1}^N
(t_i\mu^{(1)}(\sigma\ast\omega)+p_0t_ip_{\sigma\ast\omega})\\&=&\sum_{\omega\in\Omega_{l-1}}
\mu^{(1)}(\sigma\ast\omega)+p_0p_\sigma\sum_{\omega\in\Omega_{l-1}}p_\omega\\&=&\sum_{\tau\in\Gamma(\sigma,l-1)}
\mu^{(1)}(\tau)+p_0p_\sigma(1-p_0)^{l-1}.
\end{eqnarray*}
From this, we conveniently obtain
\begin{eqnarray*}
\sum_{\omega\in\Gamma(\sigma,h)}\mu^{(1)}(\omega)-\mu^{(1)}(\sigma)&=&\sum_{l=2}^h\bigg(\sum_{\omega\in\Gamma(\sigma,l)}\mu^{(1)}(\omega)
-\sum_{\omega\in\Gamma(\sigma,l-1)}\mu^{(1)}(\omega)\bigg)\\&&\;\;+\bigg(\sum_{\omega\in\Gamma(\sigma,1)}\mu^{(1)}(\omega)-\mu^{(1)}(\sigma)\bigg)
\\&=&\sum_{l=2}^hp_0p_\sigma(1-p_0)^{l-1}+p_0p_\sigma\\&=&(1-p_0)(1-(1-p_0)^{h-1})p_\sigma+p_0p_\sigma.
\end{eqnarray*}
This completes the proof of the lemma.
\end{proof}
\begin{lemma}\label{g4}
There exists a constant $C_2>0$ such that
\begin{eqnarray}\label{s2}
\sup_{j\geq 1}\max\bigg\{\bigg|\sum_{\sigma\in\Lambda_j}p_\sigma\log t_\sigma\bigg|,\bigg|\sum_{\sigma\in\Lambda_j}p_\sigma\log s_\sigma\bigg|\bigg\}\leq C_2.
\end{eqnarray}
\end{lemma}
\begin{proof}
First, we show that, for $j\geq 1$, we have
\begin{eqnarray}\label{kk6}
\sum_{\sigma\in\Lambda_j}p_\sigma\leq(1-p_0)^{k_{1j}}.
\end{eqnarray}
For this, we define, for each $k\geq 1$,
\begin{eqnarray*}
J_k:=\sum_{\sigma\in\Omega_k}p_\sigma=(1-p_0)^k;\;\;\xi(\sigma):=J_k^{-1}p_\sigma,\;\sigma\in\Omega_k.
\end{eqnarray*}
Then for each $h\geq 1$, we have
\begin{eqnarray*}
\sum_{\tau\in\Gamma(\sigma,h)}\xi(\tau)&=&\sum_{\tau\in\Gamma(\sigma,h)}J_{k+h}^{-1}p_\tau=J_{k+h}^{-1}\sum_{\omega\in\Omega_h}p_\sigma p_\omega\\&=&J_{k+h}^{-1}p_\sigma(1-p_0)^h=J_k^{-1}p_\sigma=\xi(\sigma).
\end{eqnarray*}
It follows that
\begin{eqnarray*}
\sum_{\sigma\in\Lambda_j}J_{|\sigma|}^{-1}p_\sigma=\sum_{\sigma\in\Lambda_j}\xi(\sigma)=
\sum_{\sigma\in\Lambda_j}\sum_{\tau\in\Gamma(\sigma,k_{2j}-|\sigma|)}\xi(\tau)=\sum_{\tau\in\Omega_{k_{2j}}}\xi(\tau)=1.
\end{eqnarray*}
Hence, $\sum_{\sigma\in\Lambda_j}p_\sigma\leq\max_{k_{1j}\leq k\leq k_{2j}}J_k=(1-p_0)^{k_{1j}}$. Using this and (\ref{s9}), we deduce
\begin{eqnarray*}
&&\max\bigg\{\bigg|\sum_{\sigma\in\Lambda_j}p_\sigma\log t_\sigma\bigg|,\bigg|\sum_{\sigma\in\Lambda_j}p_\sigma\log s_\sigma\bigg|\bigg\}\\&&=\max\bigg\{\sum_{\sigma\in\Lambda_j}|p_\sigma\log t_\sigma|,\sum_{\sigma\in\Lambda_j}|p_\sigma\log s_\sigma|\bigg\}\\&&\leq \sum_{\sigma\in\Lambda_j}p_\sigma\max\{\log\underline{t}^{-k_{2j}},\log\underline{s}^{-k_{2j}}\}\\&&\leq k_{2j}(1-p_0)^{k_{1j}}\max\{\log\underline{t}^{-1},\log\underline{s}^{-1}\}\\&&\leq B_1k_{1j}(1-p_0)^{k_{1j}}\max\{\log\underline{t}^{-1},\log\underline{s}^{-1}\}\to 0\;(j\to\infty).
\end{eqnarray*}
Thus, there exist a constant $C_2>0$ fulfilling (\ref{s2}). The lemma follows.
\end{proof}
\begin{lemma}\label{g13}
There exists a constant $C_3>0$ such that
\begin{eqnarray*}
\big|\sum_{\sigma\in\Lambda_j}\mu(E_\sigma)\log s_\sigma\big|\geq C_3 j\;\;{\rm for\;all\;large}\;j.
\end{eqnarray*}
\end{lemma}
\begin{proof}
By Lemma \ref{g11} and (\ref{s13}), we have
\begin{eqnarray*}
&&\sum_{\sigma\in\Lambda_j}\mu(E_\sigma)\log s_\sigma=\sum_{\sigma\in\Lambda_j}(\mu^{(1)}(\sigma)+p_\sigma)\log s_\sigma\leq\sum_{\sigma\in\Lambda_j}\mu^{(1)}(\sigma)\log s_\sigma\\&&\leq\sum_{\sigma\in\Lambda_j}\mu^{(1)}(\sigma)\log \overline{s}^{|\sigma|}\leq\sum_{\sigma\in\Lambda_j}\mu^{(1)}(\sigma)\log \overline{s}^{k_{1j}}=(1-\sum_{\sigma\in\Lambda_j}p_\sigma)\log \overline{s}^{k_{1j}}.
\end{eqnarray*}
Note that $\sum_{\sigma\in\Lambda_j}\mu(E_\sigma)\log s_\sigma<0$. Using (\ref{kk6}) and (\ref{s9}), we deduce
\begin{eqnarray*}
\big|\sum_{\sigma\in\Lambda_j}\mu(E_\sigma)\log s_\sigma\big|&=&\sum_{\sigma\in\Lambda_j}|\mu(E_\sigma)\log s_\sigma|\geq(1-\sum_{\sigma\in\Lambda_j}p_\sigma)|\log \overline{s}^{k_{1j}}|\\&\geq&(1-(1-p_0)^{k_{1j}})|\log \overline{s}^{k_{1j}}|\geq
p_0k_{1j}\log \overline{s}^{-1}\geq p_0j\log \overline{s}^{-1}.
\end{eqnarray*}
By setting $C_3:=p_0\log \overline{s}^{-1}$, the lemma follows.
\end{proof}

For a sequence $x_l\in\mathbb{R}, l\geq 1$, we take the convention that $\sum_{l=2}^1x_l:=0$. For every $\sigma\in\Omega^*$ and $h\geq 1$, we define
\begin{eqnarray*}
&&\Delta_1(\sigma, h):=u_0\sum_{l=2}^h\big(\mu^{(1)}(\sigma)+p_\sigma(1-p_0)(1-(1-p_0)^{l-2})+p_0p_\sigma\big)\\&&\;\;\;\;\;\;\;\;\;\;\;\;\;\;\;\;\;\;+u_0(\mu^{(1)}(\sigma)
+p_0p_\sigma)\\&&\;\;\;\;\;\;\;\;\;\;\;\;\;\;\;\;\;\;+p_\sigma\log t_\sigma (1-p_0)(1-(1-p_0)^{h-1})+p_0p_\sigma\log t_\sigma\\&&\;\;\;\;\;\;\;\;\;\;\;\;\;\;\;\;\;\;
+p_0p_\sigma\sum_{l=2}^h(l-1)(1-p_0)^{l-2}\sum_{i=1}^Np_i\log t_i\\&&\;\;\;\;\;\;\;\;\;\;\;\;\;\;\;\;\;\;+u_0p_\sigma(1-p_0)(1-(1-p_0)^{h-1}).
\end{eqnarray*}
We define $\Delta_2(\sigma, h)$ analogously by replacing $u_0,\log t_\sigma,\log t_i$, in the definition of $\Delta_1(\sigma, h)$ with $l_0,\log s_\sigma,\log s_i$.
The above two quantities enable us to describe the hereditary properties of "$\mu^{(1)}(\sigma)\log t_\sigma$" and "$\mu^{(1)}(\sigma)\log s_\sigma$" over the descendants of $\sigma$. We will use such hereditary properties to compare the numerator (denominator) of $\eta_j$ and that of $d_{k_{2j}}$, so that we are able to obtain the convergence order of $(\eta_j)_{j=1}^\infty$.
\begin{lemma}\label{s33}
For $\sigma\in\Omega^*$ and $h\geq 1$, we have
\begin{eqnarray*}
\sum_{\tau\in\Gamma(\sigma,h)}\mu^{(1)}(\tau)\log t_\tau=\mu^{(1)}(\sigma)\log t_\sigma+\Delta_1(\sigma, h);\\
\sum_{\tau\in\Gamma(\sigma,h)}\mu^{(1)}(\tau)\log s_\tau=\mu^{(1)}(\sigma)\log s_\sigma+\Delta_2(\sigma, h).
\end{eqnarray*}
\end{lemma}
\begin{proof}
For $h=1$, by (\ref{s13}), we have
\begin{eqnarray}\label{s1}
&&\sum_{\tau\in\Gamma(\sigma,1)}\mu^{(1)}(\tau)\log t_\tau=\sum_{i=1}^N(t_i\mu^{(1)}(\sigma)+p_0t_ip_\sigma)\log t_{\sigma\ast i}\\&&\;\;=\mu^{(1)}(\sigma)\log t_\sigma+\mu^{(1)}(\sigma)\sum_{i=1}^Nt_i\log t_i+
p_0p_\sigma\log t_\sigma+p_0p_\sigma\sum_{i=1}^Nt_i\log t_i\nonumber\\&&\;\;=\mu^{(1)}(\sigma)\log t_\sigma+\mu^{(1)}(\sigma)u_0+
p_0p_\sigma\log t_\sigma+p_0p_\sigma u_0.
\end{eqnarray}
Thus, the lemma is true for $h=1$. Next we assume that $h\geq2$. For every $l\geq 2$,
\begin{eqnarray*}
&&\sum_{\tau\in\Gamma(\sigma,l)}\mu^{(1)}(\tau)\log t_\tau=\sum_{\omega\in\Gamma(\sigma,l-1)}\sum_{i=1}^N(t_i\mu^{(1)}(\omega)+p_0t_ip_\omega)\log t_{\omega\ast i}\\&&=\sum_{\omega\in\Gamma(\sigma,l-1)}\mu^{(1)}(\omega)\log t_\omega+u_0\sum_{\omega\in\Gamma(\sigma,l-1)}\mu^{(1)}(\omega)\\&&\;\;+\sum_{\omega\in\Gamma(\sigma,l-1)}p_0p_\omega\log t_\omega+u_0\sum_{\omega\in\Gamma(\sigma,l-1)}p_0p_\omega.
\end{eqnarray*}
Applying this equality to every $2\leq l\leq h$ and using (\ref{s1}), we deduce
\begin{eqnarray}\label{ss4}
&&\sum_{\tau\in\Gamma(\sigma,h)}\mu^{(1)}(\tau)\log t_\tau\nonumber\\&&\;\;=\sum_{l=2}^h\bigg(\sum_{\omega\in\Gamma(\sigma,l)}\mu^{(1)}(\omega)\log t_\omega
-\sum_{\omega\in\Gamma(\sigma,l-1)}\mu^{(1)}(\omega)\log t_\omega\bigg)\nonumber\\&&\;\;\;\;+\bigg(\sum_{\omega\in\Gamma(\sigma,1)}\mu^{(1)}(\omega)\log t_\omega-\mu^{(1)}(\sigma)\log t_\sigma\bigg)+\mu^{(1)}(\sigma)\log t_\sigma\nonumber\\&&
=u_0\sum_{l=2}^h\sum_{\omega\in\Gamma(\sigma,l-1)}\mu^{(1)}(\omega)
+\sum_{l=2}^h\sum_{\omega\in\Gamma(\sigma,l-1)}p_0p_\omega\log t_\omega+u_0\sum_{l=2}^h\sum_{\omega\in\Gamma(\sigma,l-1)}p_0p_\omega\nonumber\\&&\;\;\;\;+
(\mu^{(1)}(\sigma)u_0+p_0p_\sigma\log t_\sigma+p_0p_\sigma u_0)+\mu^{(1)}(\sigma)\log t_\sigma.
\end{eqnarray}
By Lemma \ref{g1}, for $2\leq l\leq h$, we have
\begin{eqnarray}\label{ss6}
\sum_{\omega\in\Gamma(\sigma,l-1)}\mu^{(1)}(\omega)=\big(\mu^{(1)}(\sigma)+p_\sigma(1-p_0)(1-(1-p_0)^{l-2})+p_0p_\sigma\big).
\end{eqnarray}
Using (\ref{gg4}) and the fact that $\sum_{\sigma\in\Omega_h}p_\sigma=(1-p_0)^h$, we easily see
\begin{eqnarray}\label{ss5}
&&\sum_{l=2}^h\sum_{\omega\in\Gamma(\sigma,l-1)}p_0p_\omega\log t_\omega=\sum_{l=2}^h\sum_{\rho\in\Omega_{l-1}}p_0p_\sigma p_\rho(\log t_\sigma+\log t_\rho)\nonumber\\&&\;\;=p_0p_\sigma \log t_\sigma\sum_{l=2}^h\sum_{\rho\in\Omega_{l-1}}p_\rho+p_0p_\sigma \sum_{l=2}^h\sum_{\rho\in\Omega_{l-1}}p_\rho\log t_\rho\nonumber\\&&\;\;=p_0p_\sigma \log t_\sigma\sum_{l=2}^h(1-p_0)^{l-1}+p_0p_\sigma\sum_{l=2}^h(l-1)(1-p_0)^{l-2}\sum_{i=1}^Np_i\log t_i\nonumber\\&&\;\;=p_\sigma \log t_\sigma(1-p_0)(1-(1-p_0)^{h-1})+p_0p_\sigma\sum_{l=2}^h(l-1)(1-p_0)^{l-2}\sum_{i=1}^Np_i\log t_i;
\\&&\;\;
\sum_{l=2}^h\sum_{\omega\in\Gamma(\sigma,l-1)}p_0p_\omega=\sum_{l=2}^h\sum_{\rho\in\Omega_{l-1}}p_0p_\sigma p_\rho=(1-p_0)(1-(1-p_0)^{h-1})p_\sigma\label{ss7}.
\end{eqnarray}
By (\ref{ss4})-(\ref{ss7}), the first equality follows. Analogously, one can show the second.
\end{proof}

Corresponding to the summands in the definitions of $\Delta_1(\sigma, h)$ and $\Delta_2(\sigma, h)$, we define the following quantities which will appear in the expression of $d_{k_{2j}}$:
\begin{eqnarray*}
&&I_0:=\sum_{\sigma\in\Lambda_j\setminus\Omega_{k_{2j}}}\bigg(\sum_{l=2}^{k_{2j}-|\sigma|}\big(\mu^{(1)}(\sigma)
+p_\sigma(1-p_0)(1-(1-p_0)^{l-2})+p_0p_\sigma\big)\\&&\;\;\;\;\;\;\;\;\;\;\;\;\;\;\;\;\;\;\;\;\;\;\;\;\;\;\;\;\;\;\;\;\;\;\;\;
+\mu^{(1)}(\sigma)+p_0p_\sigma\bigg);
\nonumber\\&&I_1:=\sum_{\sigma\in\Lambda_j\setminus\Omega_{k_{2j}}}\bigg(p_\sigma\log t_\sigma (1-p_0)(1-(1-p_0)^{k_{2j}-|\sigma|-1})+p_0p_\sigma\log t_\sigma\bigg);\label{s3'}\\
&&I_2:=\sum_{\sigma\in\Lambda_j\setminus\Omega_{k_{2j}}}p_0p_\sigma\sum_{l=2}^{k_{2j}-|\sigma|}(l-1)(1-p_0)^{l-2}\sum_{i=1}^Np_i\log t_i;\label{s3}\\
&&I_3:=\sum_{\sigma\in\Lambda_j\setminus\Omega_{k_{2j}}}p_\sigma(1-p_0)(1-(1-p_0)^{k_{2j}-|\sigma|-1}).\nonumber
\end{eqnarray*}
We define $\widetilde{I}_1$ by replacing $\log t_\sigma$ in the definition of $I_1$ with $\log s_\sigma$ and define $\widetilde{I}_2$ by replacing $\log t_i$ in the definition of $I_2$ with $\log s_i$. Set
\[
B_2:=\max\bigg\{\sum_{i=1}^N|p_i\log t_i|,\sum_{i=1}^N|p_i\log s_i|\bigg\}\cdot\sum_{l=2}^\infty(l-1)(1-p_0)^{l-2}.
\]
Clearly, $B_2<\infty$. We have
\begin{lemma}\label{s32}
For $I_i,0\leq i\leq 4$, we have
\begin{eqnarray*}
|I_0|\leq B_1j;\;\max\{|I_1|,|\widetilde{I}_1|\}\leq C_2;\;\max\{|I_2|,|\widetilde{I}_2|\}\leq B_2;\;|I_3|\leq 1.
\end{eqnarray*}
\end{lemma}
\begin{proof}
By the definition of $I_0$, we have
\begin{eqnarray*}
|I_0|&\leq& \sum_{\sigma\in\Lambda_j}\bigg(\sum_{l=2}^{k_{2j}-|\sigma|}\big(\mu^{(1)}(\sigma)+p_\sigma\big)+(\mu^{(1)}(\sigma)+p_\sigma)\bigg)
\\&\leq&\sum_{\sigma\in\Lambda_j}(k_{2j}-k_{1j})\mu(E_\sigma)
=(k_{2j}-k_{1j})\sum_{\sigma\in\Lambda_j}\mu(E_\sigma)\leq
B_1j.
\end{eqnarray*}
By (\ref{s3'}) and Lemma \ref{g4}, we have
\begin{eqnarray*}
\max\{|I_1|,|\widetilde{I}_1|\}\leq(1-p_0^2)\max\big\{\sum_{\sigma\in\Lambda_j}|p_\sigma\log t_\sigma|,\sum_{\sigma\in\Lambda_j}|p_\sigma\log s_\sigma|\big\}\leq C_2.
\end{eqnarray*}
Also, by the definitions of $I_2,\widetilde{I}_2$, we have
\begin{eqnarray*}
\max\{|I_2|,|\widetilde{I}_2|\}\leq B_2\sum_{\sigma\in\Lambda_j}p_0p_\sigma\leq B_2(1-p_0)^{k_{1j}}\leq B_2.
\end{eqnarray*}
Finally, by (\ref{kk6}) and the definition of $I_3$, we conclude
\begin{eqnarray*}
|I_3|\leq(1-p_0)\sum_{\sigma\in\Lambda_j}p_\sigma\leq (1-p_0)^{k_{1j}+1}\leq 1.
\end{eqnarray*}
This complete the proof of the lemma.
\end{proof}

\begin{lemma}\label{g3}
There exists constant $C_4>0$ such that $|\eta_j-d_0|\leq C_4 j^{-1}$.
\end{lemma}
\begin{proof}
Let $x_j,y_j$ denote the numerator and denominator of $\eta_j$.
Note that
\[
\Omega_{k_{2j}}=(\Lambda_j\cap\Omega_{k_{2j}})\cup\big(\bigcup_{\sigma\in\Lambda_j\setminus\Omega_{k_{2j}}}\Gamma(\sigma,k_{2j}-|\sigma|)\big).
\]
By Lemma \ref{s33}, we deduce
\begin{eqnarray*}
I_4:&=&\sum_{\tau\in\Omega_{k_{2j}}}\mu^{(1)}(\tau)\log t_\tau
\\&=&\sum_{\sigma\in\Lambda_j\setminus\Omega_{k_{2j}}}\sum_{\tau\in\Gamma(\sigma,k_{2j}-|\sigma|)}\mu^{(1)}(\tau)\log t_\tau+\sum_{\sigma\in\Lambda_j\cap\Omega_{k_{2j}}}\mu^{(1)}(\sigma)\log t_\sigma\\&=&\sum_{\sigma\in\Lambda_j\setminus\Omega_{k_{2j}}}\big(\mu^{(1)}(\sigma)\log t_\sigma+\Delta_1(\sigma, k_{2j}-|\sigma|)\big)+\sum_{\sigma\in\Lambda_j\cap\Omega_{k_{2j}}}\mu^{(1)}(\sigma)\log t_\sigma\\&=&\sum_{\sigma\in\Lambda_j}\mu^{(1)}(\sigma)\log t_\sigma+I_0u_0+I_1+I_2+I_3u_0\\&=&\big(x_j-\sum_{\sigma\in\Lambda_j}p_\sigma\log t_\sigma\big)+I_0u_0+I_1+I_2+I_3u_0.
\end{eqnarray*}
In an analogous manner, we have
\begin{eqnarray*}
I_5:=\sum_{\tau\in\Omega_{k_{2j}}}\mu^{(1)}(\tau)\log s_\tau
=y_j-\sum_{\sigma\in\Lambda_j}p_\sigma\log s_\sigma+I_0l_0+\widetilde{I}_1+\widetilde{I}_2+I_3l_0.
\end{eqnarray*}
For convenience, we write
\begin{eqnarray*}
&&I_6:=-\sum_{\sigma\in\Lambda_j}p_\sigma\log t_\sigma+I_1+I_2+I_3u_0+\sum_{\sigma\in\Omega_{k_{2j}}}p_\sigma\log t_\sigma;\\
&&I_7:=-\sum_{\sigma\in\Lambda_j}p_\sigma\log s_\sigma+\widetilde{I}_1+\widetilde{I}_2+I_3l_0+\sum_{\sigma\in\Omega_{k_{2j}}}p_\sigma\log s_\sigma.
\end{eqnarray*}
By Lemmas \ref{s2}, \ref{s32}, we know that $I_6,I_7$ are both bounded. Hence, there exists a constant $A_3>0$ such that $|I_7|,|I_6l_0-I_7u_0|\leq A_3$.
By (\ref{s13}), we have
\begin{eqnarray}
d_{k_{2j}}=\frac{I_4+\sum_{\tau\in\Omega_{k_{2j}}}p_\tau\log t_\tau}{I_5+\sum_{\tau\in\Omega_{k_{2j}}}p_\tau\log s_\tau}\label{zz5}=\frac{x_j+I_0u_0+I_6}{y_j+I_0l_0+I_7}\label{ss1}.
\end{eqnarray}
By Lemmas \ref{s2}, \ref{s32}, we have
\begin{eqnarray*}
|y_j|\geq C_3 j,\;\;|I_0|\leq B_1j.
\end{eqnarray*}
This implies that, for $j\geq A_3$, we have
\[
|y_j+I_0l_0+I_7|\leq(C_3+B_1|l_0|+1)j\leq (C_3+B_1|l_0|+1)C_3^{-1}|y_j|=:B_3|y_j|.
\]
Note that $y_j<0$ and $I_0l_0<0$. Hence, for $j\geq 2A_3/C_3$, we have
\[
|y_j+I_0l_0+I_7|\geq|y_j+I_0l_0|-A_3\geq|y_j|-A_3\geq 2^{-1}C_3j.
\]
From this and Lemma \ref{g0}, we deduce
\begin{eqnarray*}
&&\frac{C_1}{B_1 j}\geq\frac{C_1}{k_{2j}}\geq|d_{k_{2j}}-d_0|=\bigg|\frac{x_j+I_0u_0+I_6}{y_j+I_0l_0+I_7}-\frac{u_0}{l_0}\bigg|\\&&=\bigg|\frac{x_j l_0-y_j u_0+I_6l_0-I_7u_0}{(y_j+I_0l_0+I_7)l_0}\bigg|\geq\bigg|\frac{x_j l_0-y_j u_0}{B_3y_j l_0}\bigg|-\bigg|\frac{2A_3}{C_3j l_0}\bigg|.
\end{eqnarray*}
Thus, there exists a constant $C_4>0$ such that, for all large $j$, we have
\begin{eqnarray*}
|\eta_j-d_0|=\bigg|\frac{x_j}{y_j}-\frac{u_0}{l_0}\bigg|=\bigg|\frac{x_j l_0-y_j u_0}{y_j l_0}\bigg|\leq C_4 j^{-1}.
\end{eqnarray*}
This completes the proof of the lemma.
\end{proof}

\emph{Proof of Theorem \ref{mthm} for ISMs in Case I}

By proposition \ref{g12}, it suffices to show that $\overline{Q}_0^{d_0}(\mu)<\infty$.
For each $j\geq 1$ and every $\sigma\in\Lambda_j$, we choose an arbitrary point $a_\sigma\in E_\sigma$. Then
\[
\hat{e}_{\phi_j}(\mu)\leq\sum_{\sigma\in\Lambda_j}\int_{E_\sigma}\log d(x,a_\sigma)d\mu(x)\leq\sum_{\sigma\in\Lambda_j}\mu(E_\sigma)\log s_\sigma.
\]
Using this, (\ref{s8}) and (\ref{s9}), we further deduce
\begin{eqnarray*}
&&\frac{1}{d_0}\log\phi_j+\hat{e}_{\phi_j}(\mu)\leq\frac{1}{d_0}\log\phi_j+\sum_{\sigma\in\Lambda_j}\mu(E_\sigma)\log s_\sigma\\&&=\frac{1}{d_0}\log\phi_j+\frac{1}{\eta_j}\sum_{\sigma\in\Lambda_j}\mu(E_\sigma)\log t_\sigma\leq\frac{1}{d_0}\log\underline{t}^{-(j+1)}+\frac{1}{\eta_j}\log \underline{t}^{k_{1j}}\\&&\leq\frac{1}{d_0}\log\underline{t}^{-(j+1)}+\frac{1}{\eta_j}\log \underline{t}^j\\&&=d_0^{-1}\log\underline{t}^{-1}+j(d_0^{-1}-\eta_j^{-1})\log\underline{t}^{-1}
\\&&=d_0^{-1}\log\underline{t}^{-1}+j(d_0\eta_j)^{-1}(\eta_j-d_0)\log\underline{t}^{-1}.
\end{eqnarray*}
Not that $\eta_j\to d_0$ as $j\to\infty$. Thus, for large $j$, by Lemma \ref{g3}, we have
\begin{eqnarray*}
d_0^{-1}\log\phi_j+\hat{e}_{\phi_j}(\mu)&\leq& d_0^{-1}\log\underline{t}^{-1}+2jd_0^{-2}(\eta_j-d_0)\log\underline{t}^{-1}\\
&\leq&d_0^{-1}\log\underline{t}^{-1}+2d_0^{-2}C_4\log\underline{t}^{-1}.
\end{eqnarray*}
By this and Lemma \ref{interim}, we conclude that $\overline{Q}_0^{d_0}(\mu)<\infty$. The theorem follows.

\section{The upper quantization coefficient of ISMs in Case II}
In this section, $\mu$ denotes an ISM in Case II. Without loss of generality, we assume that $|K|=1$. Then
\[
|C|\leq|K|=1;\;\; |f_\sigma(K)|=s_\sigma,\;\;|f_\sigma(C_\rho)|=s_\sigma c_\rho|C|,\;\;\sigma\in\Omega^*,\;\rho\in\Phi^*.
\]
Let $\underline{p}:=\min_{1\leq i\leq N}p_i$ and $\underline{t}:=\min_{1\leq i\leq M}t_i$.
For every $j\geq 1$, we define
\begin{eqnarray}
&&\underline{q}:=\min\{\underline{p},\underline{t}\};\;\;\Gamma_j:=\{\sigma\in\Omega^*:p_{\sigma^-}\geq\underline{q}^j>p_\sigma\};\label{zz1}\\&&\psi_j:={\rm card}(\Gamma_j);\;\;
l_{1j}:=\min_{\sigma\in\Gamma_j}|\sigma|,\;l_{2j}:=\max_{\sigma\in\Gamma_j}|\sigma|.\label{zz2}
\end{eqnarray}
Set $\overline{p}:=\max\{\max_{1\leq i\leq N}p_i,\max_{1\leq i\leq M}t_i\}$. As we did for (\ref{s9}), one can easily see
\begin{eqnarray}\label{s16}
j\leq l_{1j}\leq l_{2j}\leq 2\log\underline{q}\;(\log\overline{p})^{-1}j.
\end{eqnarray}
\begin{remark}\label{gg3}
For every $0\leq h\leq l_{1j}-1$ and $\sigma\in\Omega_h$, we have, $p_\sigma\geq\underline{q}^j$, otherwise, $\min_{\sigma\in\Gamma_j}|\sigma|<l_{1j}$, a contradiction. Also, every $\sigma\in \Lambda_{\Gamma_j}^*$ has a proper descendant $\tau\in\Gamma_j$. This implies that $p_\sigma\geq p_{\tau^-}\geq\underline{q}^j$. For every $j\geq1$, we write
\[
\Psi_j:=\bigg(\bigcup_{h=0}^{l_{1j}-1}\Omega_h\bigg)\cup\Lambda_{\Gamma_j}^*.
\]
\end{remark}
For every $\sigma\in\Psi_j$, by Remark \ref{gg3}, we may define
\begin{eqnarray}\label{zz3}
\Gamma_j(\sigma):=\{\rho\in\Phi^*:p_\sigma t_{\rho^-}\geq\underline{q}^j>p_\sigma t_\rho\};\;\;\psi_j(\sigma):={\rm card}(\Gamma_j(\sigma)).
\end{eqnarray}
 With $\psi_j$ as defined in  (\ref{zz2}), we write
\begin{eqnarray}\label{zz6}
M_j:=\psi_j+\sum_{\sigma\in\Psi_j}\psi_j(\sigma).
\end{eqnarray}
\begin{lemma}\label{g9}
There exist constants $N_1,N_2>0$ such that
\[
p_0\underline{q}^{-j}\leq M_j\leq N_1\underline{q}^{-j};\;M_j\leq M_{j+1}\leq N_2M_j;\;\;j\geq 1.
\]
\end{lemma}
\begin{proof}
For every $j\geq 1$, we write
\begin{eqnarray*}
Q_j:=\sum_{h=0}^{l_{1j}-1}\sum_{\sigma\in\Omega_h}\sum_{\rho\in\Gamma_j(\sigma)}p_0p_\sigma t_\rho+\sum_{\sigma\in\Lambda_{\Gamma_j}^*}\sum_{\rho\in\Gamma_j(\sigma)}p_0p_\sigma t_\rho+\sum_{\sigma\in\Gamma_j}p_\sigma
\end{eqnarray*}
Note that $\sum_{\rho\in\Gamma_j(\sigma)}t_\rho=1$ and $\Lambda_{\Gamma_j}^*\subset\bigcup_{h=l_{1j}}^{l_{2j}-1}\Omega_h$. We deduce
\begin{eqnarray*}
Q_j&=&\sum_{h=0}^{l_{1j}-1}\sum_{\sigma\in\Omega_h}p_0p_\sigma+\sum_{\sigma\in\Lambda_{\Gamma_j}^*}p_0p_\sigma
+\sum_{\sigma\in\Gamma_j}p_\sigma\\&\leq&\sum_{h=0}^{l_{1j}-1}\sum_{\sigma\in\Omega_h}p_0p_\sigma
+\sum_{h=l_{1j}}^{l_{2j}-1}\sum_{\sigma\in\Omega_h}p_0p_\sigma+
\sum_{\sigma\in\Gamma_j}p_\sigma\\&\leq&\sum_{h=0}^{l_{1j}-1}p_0(1-p_0)^h+\sum_{h=l_{1j}}^{l_{2j}-1}p_0(1-p_0)^h
+(1-p_0)^{l_{1j}}\;\;{\rm by}\; (\ref{kk6})\\&\leq&1-(1-p_0)^{l_{2j}}+(1-p_0)^{l_{1j}}\leq 2-p_0.
\end{eqnarray*}
In addition, we have $Q_j\geq\sum_{h=0}^{l_{1j}-1}\sum_{\sigma\in\Omega_h}p_0p_\sigma\geq p_0$.
By (\ref{zz1}) and (\ref{zz3}), we deduce
\begin{eqnarray*}
M_j p_0\underline{q}^{j+1}\leq Q_j\leq M_j\underline{q}^j.
\end{eqnarray*}
Combing the above analysis, we have
\[
M_j p_0\underline{q}^{j+1}\leq Q_j\leq2-p_0; p_0\leq Q_j\leq M_j\underline{q}^j.
\]
Hence, the lemma follows by setting
\[
N_1:=p_0^{-1}\underline{q}^{-1}(2-p_0),\; N_2:=N_1p_0^{-1}\underline{q}^{-1}.
\]\end{proof}
For convenience, we write
\begin{eqnarray*}
&&T_j(1):=\sum_{h=0}^{l_{1j}-1}\sum_{\sigma\in\Omega_h}\sum_{\rho\in\Gamma_j(\sigma)}p_0p_\sigma t_\rho\log(p_0p_\sigma t_\rho);\\
&&T_j(2):=\sum_{\sigma\in\Lambda_{\Gamma_j}^*}\sum_{\rho\in\Gamma_j(\sigma)}p_0p_\sigma t_\rho\log(p_0p_\sigma t_\rho);\;\;
T_j(3):=\sum_{\sigma\in\Gamma_j}p_\sigma\log p_\sigma.
\end{eqnarray*}
Analogously, for every $j\geq 1$, we write
\begin{eqnarray*}
&&R_j(1):=\sum_{h=0}^{l_{1j}-1}\sum_{\sigma\in\Omega_h}\sum_{\rho\in\Gamma_j(\sigma)}p_0p_\sigma t_\rho\log(s_\sigma c_\rho|C|);\\
&&R_j(2):=\sum_{\sigma\in\Lambda_{\Gamma_j}^*}\sum_{\rho\in\Gamma_j(\sigma)}p_0p_\sigma t_\rho\log(s_\sigma c_\rho|C|);\;\;
R_j(3):=\sum_{\sigma\in\Gamma_j}p_\sigma\log s_\sigma.
\end{eqnarray*}
In order to estimate the upper quantization coefficient, we need to consider the convergence order of the following sequence which is connected with the geometric mean error of $\mu$ (see (\ref{s10})):
\[
\xi_j:=\frac{T_j(1)+T_j(2)+T_j(3)}{R_j(1)+R_j(2)+R_j(3)},\;\;j\geq 1.
\]
The main idea is to compare the numerator and denominator of $\xi_j$ and those of $d_0$.
With Lemmas \ref{g5}-\ref{g7}, we subtract from $T_j(h),R_j(h),h=1,2,3$, the summands that are relevant to $d_0$ and control the corresponding differences. By doing so, we will finally be able to estimate the difference between $\xi_j$ and $d_0$ (see Lemma \ref{g8}).
\begin{lemma}\label{g5}
There exists a constant $C_5$ such that $|e_j|,|\widetilde{e}_j|\leq C_5$, where
\begin{eqnarray*}
e_j:=T_j(1)-\sum_{h=0}^{l_{1j}-1}\sum_{\sigma\in\Omega_h}\sum_{\rho\in\Gamma_j(\sigma)}p_0p_\sigma t_\rho\log t_\rho;\\
\widetilde{e}_j:=R_j(1)-\sum_{h=0}^{l_{1j}-1}\sum_{\sigma\in\Omega_h}\sum_{\rho\in\Gamma_j(\sigma)}p_0p_\sigma t_\rho\log c_\rho.
\end{eqnarray*}
\end{lemma}
\begin{proof}
Note that $\sum_{\rho\in\Gamma_j(\sigma)}t_\rho=1$. We have
\begin{eqnarray*}
T_j(1,a):&=&\sum_{h=0}^{l_{1j}-1}\sum_{\sigma\in\Omega_h}\sum_{\rho\in\Gamma_j(\sigma)}p_0p_\sigma t_\rho\log p_0=\sum_{h=0}^{l_{1j}-1}\sum_{\sigma\in\Omega_h}p_0p_\sigma \log p_0\\&=&p_0\log p_0\sum_{h=0}^{l_{1j}-1}\sum_{\sigma\in\Omega_h}p_\sigma=p_0\log p_0\sum_{h=0}^{l_{1j}-1}(1-p_0)^h\\&=&(1-(1-p_0)^{l_{1j}})\log p_0.
\end{eqnarray*}
In a similar manner, we deduce
\begin{eqnarray*}
T_j(1,b):&=&\sum_{h=0}^{l_{1j}-1}\sum_{\sigma\in\Omega_h}\sum_{\rho\in\Gamma_j(\sigma)}p_0p_\sigma t_\rho\log p_\sigma=\sum_{h=0}^{l_{1j}-1}\sum_{\sigma\in\Omega_h}p_0p_\sigma \log p_\sigma\\&=&p_0\sum_{h=0}^{l_{1j}-1}\sum_{\sigma\in\Omega_h}p_\sigma\log p_\sigma=p_0\sum_{h=0}^{l_{1j}-1}h(1-p_0)^{h-1}\sum_{i=1}^Np_i\log p_i\\&=&p_0\sum_{i=1}^Np_i\log p_i \sum_{h=0}^{l_{1j}-1}h(1-p_0)^{h-1}.
\end{eqnarray*}
Since $\sum_{h=0}^\infty h(1-p_0)^{h-1}<\infty$, there exists a constant $C_5(1)>0$ such that
$$
|e_j|=|T_j(1,a)+T_j(1,a)|=|T_j(1,a)|+|T_j(1,a)|\leq C_5(1).
$$
Analogously, one can show that, $|\widetilde{e}_j|\leq C_5(2)$ for some constant $C_5(2)>0$. The lemma follows by setting $C_5:=C_5(1)+C_5(2)$.
\end{proof}
\begin{lemma}\label{g6}
There exists a constant $C_6$ such that $|\beta_j|,|\widetilde{\beta}_j|\leq C_6$, where
\begin{eqnarray*}
\beta_j:=T_j(2)-\sum_{\sigma\in\Lambda_{\Gamma_j}^*}\sum_{\rho\in\Gamma_j(\sigma)}p_0p_\sigma t_\rho\log t_\rho;\\
\widetilde{\beta}_j:=R_j(2)-\sum_{\sigma\in\Lambda_{\Gamma_j}^*}\sum_{\rho\in\Gamma_j(\sigma)}p_0p_\sigma t_\rho\log c_\rho.
\end{eqnarray*}
\end{lemma}
\begin{proof}
Note that $\sum_{\rho\in\Gamma_j(\sigma)}t_\rho=1$. We have
\begin{eqnarray*}
T_j(2,a):&=&\sum_{\sigma\in\Lambda_{\Gamma_j}^*}\sum_{\rho\in\Gamma_j(\sigma)}p_0p_\sigma t_\rho|\log p_0|=\sum_{\sigma\in\Lambda_{\Gamma_j}^*}p_0p_\sigma |\log p_0|\sum_{\rho\in\Gamma_j(\sigma)}t_\rho\\&\leq&\sum_{h=l_{1j}}^{l_{2j}-1}\sum_{\sigma\in\Omega_h}p_0p_\sigma |\log p_0|\leq p_0|\log p_0|\sum_{h=l_{1j}}^{l_{2j}-1}\sum_{\sigma\in\Omega_h}p_\sigma\\&=&p_0|\log p_0|\sum_{h=l_{1j}}^{l_{2j}-1}(1-p_0)^h\leq(1-p_0)^{l_{1j}}|\log p_0|\leq|\log p_0|.
\end{eqnarray*}
In a similar manner, we deduce
\begin{eqnarray*}
T_j(2,b):&=&\sum_{\sigma\in\Lambda_{\Gamma_j}^*}\sum_{\sigma\in\Omega_h}\sum_{\rho\in\Gamma_j(\sigma)}p_0p_\sigma t_\rho|\log p_\sigma|=\sum_{\sigma\in\Lambda_{\Gamma_j}^*}p_0p_\sigma |\log p_\sigma|\\&\leq&p_0\sum_{h=l_{1j}}^{l_{2j}-1}\sum_{\sigma\in\Omega_h}p_\sigma|\log p_\sigma|=p_0\sum_{h=l_{1j}}^{l_{2j}-1}h(1-p_0)^{h-1}\sum_{i=1}^Np_i|\log p_i|\\&=&p_0\sum_{i=1}^Np_i|\log p_i| \sum_{h=l_{1j}}^{l_{2j}-1}h(1-p_0)^{h-1}.
\end{eqnarray*}
Since $\sum_{h=0}^\infty h(1-p_0)^{h-1}<\infty$, there exists a constant $C_6(1)>0$ such that
$$
|\beta_j|=|T_j(2,a)+T_j(2,b)|=|T_j(2,a)|+|T_j(2,b)|\leq C_6(1).
$$
Analogously, one can show that, $|\widetilde{\beta}_j|\leq C_6(2)$ for some constant $C_6(2)>0$. The lemma follows by setting $C_6:=C_6(1)+C_6(2)$.
\end{proof}
\begin{lemma}\label{g7}
There exists a constant $C_7$ such that $|\chi_j|,|\widetilde{\chi}_j|\leq C_7$, where
\[
\chi_j:=\sum_{\sigma\in\Gamma_j}p_\sigma \log p_\sigma,\;\widetilde{\chi}_j:=\sum_{\sigma\in\Gamma_j}p_\sigma \log s_\sigma.
\]
\end{lemma}
\begin{proof}
This can be shown in the same way as we did for Lemma \ref{g4}.
\end{proof}
\begin{lemma}\label{g10}
For every finite maximal antichain $\Gamma$ in $\Phi^*$, we have
\begin{eqnarray*}
l_0\sum_{\rho\in\Gamma}t_\rho\log t_\rho=u_0\sum_{\rho\in\Gamma}t_\rho\log c_\rho.
\end{eqnarray*}
\end{lemma}
\begin{proof}
Set $l(\Gamma):=\min_{\rho\in\Gamma}|\rho|$ and $L(\Gamma):=\max_{\rho\in\Gamma}|\rho|$. We define
\[
H_k:=\sum_{\rho\in\Phi_k}t_\rho\log t_\rho,\;\;b(\rho):=t_\rho(\log t_\rho-H_{|\rho|}).
\]
For $\rho\in\Phi_k$ and $h\geq 1$, we have
\begin{eqnarray*}
\sum_{\omega\in\Gamma(\rho,h)}b(\omega)&=&\sum_{\omega\in\Gamma(\rho,h)}t_\omega(\log t_\omega-H_{k+h})=\sum_{\tau\in\Phi_h}t_{\rho\ast\tau}(\log t_{\rho\ast\tau}-H_{k+h})\\&=&t_\rho\log t_\rho+t_\rho\sum_{\tau\in\Phi_h}t_\tau\log t_\tau-t_\rho H_{k+h}\\&=&t_\rho\log t_\rho+t_\rho(H_{k+h}-H_k)-t_\rho H_{k+h}\\&=&t_\rho(\log t_\rho-H_k)=b(\rho).
\end{eqnarray*}
Applying this to every $\tau\in\Gamma$ with $h=L(\Gamma)-|\rho|$, we deduce
\begin{eqnarray*}
\sum_{\rho\in\Gamma}b(\rho)=\sum_{\rho\in\Gamma}\sum_{\omega\in\Gamma(\rho,L(\Gamma)-|\rho|)}b(\omega)=\sum_{\omega\in\Phi_{L(\Gamma)}}b(\rho)=0.
\end{eqnarray*}
It follows that $\sum_{\rho\in\Phi_k}t_\rho\log t_\rho=\sum_{\rho\in\Phi_k}t_\rho H_{|\rho|}$. Similarly, we have,
\[
\sum_{\rho\in\Phi_k}t_\rho\log c_\rho=\sum_{\rho\in\Phi_k}t_\rho T_{|\rho|},\; {\rm with}\;\; T_{|\rho|}:=\sum_{\rho\in\Phi_{|\rho|}}t_\rho\log c_\rho.
\]
Note that $H_k=ku_0$ and $T_k=kl_0$ for all $k\geq 1$. Thus,
\begin{eqnarray*}
\frac{\sum_{\rho\in\Gamma}t_\rho\log t_\rho}{\sum_{\rho\in\Gamma}t_\rho\log c_\rho}
=\frac{\sum_{\rho\in\Gamma}t_\rho H_{|\rho|}}{\sum_{\rho\in\Gamma}t_\rho T_{|\rho|}}=\frac{u_0}{l_0}.
\end{eqnarray*}
This completes the proof of the lemma.
\end{proof}

\begin{lemma}\label{g8}
There exists constant $C_8$ such that $|\xi_j-d_0|\leq C_8 j^{-1}$.
\end{lemma}
\begin{proof}
For every $j\geq 1$, we define
\[
h_j:=e_j+\beta_j+\chi_j,\;\widetilde{h}_j:=\widetilde{e}_j+\widetilde{\beta}_j+\widetilde{\chi}_j.
\]
For $B_4:=C_5+C_6+C_7$, by Lemmas \ref{g5}-\ref{g7}, we have $|h_j|,|\widetilde{h}_j|\leq B_4$. Note that
\begin{eqnarray*}
\xi_j=\frac{\sum_{\sigma\in\Psi_j}\sum_{\rho\in\Gamma_j(\sigma)}p_0p_\sigma t_\rho\log t_\rho+h_j}
{\sum_{\sigma\in\Psi_j}\sum_{\rho\in\Gamma_j(\sigma)}p_0p_\sigma t_\rho\log c_\rho+\widetilde{h}_j}.
\end{eqnarray*}
Note that $\Gamma_j(\sigma),\sigma\in\Psi_j$, are maximal antichains in $\Phi^*$. By Lemma \ref{g10}, we have
\begin{eqnarray*}
l_0\sum_{\rho\in\Gamma_j(\sigma)}t_\rho\log t_\rho=u_0\sum_{\rho\in\Gamma_j(\sigma)}t_\rho\log c_\rho.
\end{eqnarray*}
It follows that
\begin{eqnarray*}
\frac{\sum_{\sigma\in\Psi_j}\sum_{\rho\in\Gamma_j(\sigma)}p_0p_\sigma t_\rho\log t_\rho}
{\sum_{\sigma\in\Psi_j}\sum_{\rho\in\Gamma_j(\sigma)}p_0p_\sigma t_\rho\log c_\rho}=\frac{u_0}{l_0}.
\end{eqnarray*}
Recall that $\Gamma_j(\theta)=\{\rho\in\Phi^*:t_{\rho^-}\geq \underline{q}^j>t_\rho\}$ is a finite maximal antichain in $\Phi^*$. Set $L:=\min_{\rho\in\Gamma_j(\theta)}|\rho|$. It is easy to see that $L\geq j\log\underline{q}/\log\underline{t}$. Thus,
\begin{eqnarray*}
\bigg|\sum_{\rho\in\Gamma_j(\theta)}t_\rho\log c_\rho\bigg|\geq\bigg|\sum_{\rho\in\Phi_L}t_\rho\log c_\rho\bigg|=L\sum_{i=1}^N|t_i\log c_i|\geq j |l_0|\frac{\log\underline{q}}{\log\underline{t}}=:B_5 j.
\end{eqnarray*}
where $B_5:=|l_0|\log\underline{q}/\log\underline{t}$. Hence, for all $j>2B_4/B_5$, we have
\begin{eqnarray*}
|\xi_j-d_0|&=&\bigg|\frac{\sum_{\sigma\in\Psi_j}\sum_{\rho\in\Gamma_j(\sigma)}p_0p_\sigma t_\rho\log t_\rho+h_j}
{\sum_{\sigma\in\Psi_j}\sum_{\rho\in\Gamma_j(\sigma)}p_0p_\sigma t_\rho\log c_\rho+\widetilde{h}_j}-\frac{u_0}{l_0}\bigg|\\
&=&\frac{|h_j l_0-\widetilde{h}_j u_0|}{|\sum_{\sigma\in\Psi_j}\sum_{\rho\in\Gamma_j(\sigma)}p_0p_\sigma t_\rho\log c_\rho+\widetilde{h}_j|}\\&\leq&
\frac{B_4|l_0+u_0|}{|\sum_{\sigma\in\Psi_j}\sum_{\rho\in\Gamma_j(\sigma)}p_0p_\sigma t_\rho\log c_\rho+\widetilde{h}_j|}\\&\leq&\frac{B_4|l_0+u_0|}{|\sum_{\rho\in\Gamma_j(\theta)}p_0t_\rho\log c_\rho+B_4|}\leq (2B_4|l_0+u_0|)B_5^{-1} j^{-1}.
\end{eqnarray*}
Hence, the lemma follows by setting $C_8:=(2B_4|l_0+u_0|)B_5^{-1}$.
\end{proof}

\emph{Proof of Theorem \ref{mthm} for ISMs in Case II}

By proposition \ref{g12}, it suffices to show that $\overline{Q}_0^{d_0}(\mu)<\infty$.
For every $\sigma\in\Gamma_j$, let $b_\sigma$ be an arbitrary point in $K_\sigma$; For every $\sigma\in\Psi_j$ and $\rho\in\Gamma_j(\sigma)$, let $b_\rho$ be an arbitrary point in $C_\rho$. Then, for $M_j$ as defined in (\ref{zz6}), the cardinality of the set of all these points $b_\sigma,b_\rho$, is not greater than $M_j$. Thus, we have
\begin{eqnarray*}
\hat{e}_{M_j}(\mu)&\leq&\sum_{\sigma\in\Psi_j}\sum_{\rho\in\Gamma_j(\sigma)}\int_{f_\sigma(C_\rho)}\log d(x,b_\rho)d\mu(x)+\sum_{\sigma\in\Gamma_j}\int_{K_\sigma}\log d(x,b_\sigma)d\mu(x)\\&\leq&\sum_{\sigma\in\Psi_j}\sum_{\rho\in\Gamma_j(\sigma)}p_0p_\sigma t_\rho\log (s_\sigma c_\rho|C|)+\sum_{\sigma\in\Gamma_j}p_\sigma\log s_\sigma\\&=&\frac{1}{\xi_j}\bigg(\sum_{\sigma\in\Psi_j}\sum_{\rho\in\Gamma_j(\sigma)}p_0p_\sigma t_\rho\log (p_0p_\sigma t_\rho)+\sum_{\sigma\in\Gamma_j}p_\sigma\log p_\sigma\bigg)\\&\leq&\frac{1}{\xi_j}\sum_{\sigma\in\Psi_j}\sum_{\rho\in\Gamma_j(\sigma)}p_0p_\sigma t_\rho\log (p_\sigma t_\rho)
\end{eqnarray*}
Thus, by the definitions of $\Gamma_j(\sigma)$ (see (\ref{zz3})), we have
\begin{equation}\label{s10}
\hat{e}_{M_j}(\mu)\leq\frac{1}{\xi_j}\sum_{\sigma\in\Psi_j}\sum_{\rho\in\Gamma_j(\sigma)}p_0p_\sigma t_\rho\log \underline{q}^j\leq\frac{1}{\xi_j}\sum_{h=0}^{l_{1j}-1}\sum_{\sigma\in\Omega_h}\sum_{\rho\in\Gamma_j(\sigma)}p_0p_\sigma t_\rho\log \underline{q}^j.
\end{equation}
Now one can easily see
\begin{eqnarray*}
\sum_{h=0}^{l_{1j}-1}\sum_{\sigma\in\Omega_h}\sum_{\rho\in\Gamma_j(\sigma)}p_0p_\sigma t_\rho=\sum_{h=0}^{l_{1j}-1}\sum_{\sigma\in\Omega_h}p_0p_\sigma
=p_0\sum_{h=0}^{l_{1j}-1}(1-p_0)^h=(1-(1-p_0)^{l_{1j}}).
\end{eqnarray*}
From this and (\ref{s10}), we deduce
\begin{eqnarray*}
\hat{e}_{M_j}(\mu)\leq\frac{1}{\xi_j}(1-(1-p_0)^{l_{1j}})\log \underline{q}^j=\frac{1}{\xi_j}\log \underline{q}^j-\frac{1}{\xi_j}(1-p_0)^{l_{1j}}\log \underline{q}^j.
\end{eqnarray*}
By Lemma \ref{g8}, $\xi_j\to d_0$ as $j\to\infty$. Hence, for large $j$, we have
\begin{eqnarray*}
\hat{e}_{M_j}(\mu)\leq\frac{1}{\xi_j}\log \underline{q}^j-\frac{2j}{d_0}(1-p_0)^{l_{1j}}\log \underline{q}^j.
\end{eqnarray*}
By (\ref{s16}), we have that $l_{1j}\geq j$. So, we have, $(1-p_0)^{l_{1j}}j\to 0$ as $j\to\infty$. Thus, there exists a constant $C_9>0$ such that
$\hat{e}_{M_j}(\mu)\leq\xi_j^{-1}\log \underline{q}^j+C_9$. Using this and Lemma \ref{g9}, we deduce
\begin{eqnarray*}
d_0^{-1}\log M_j+\hat{e}_{M_j}(\mu)&\leq& d_0^{-1}\log M_j+\xi_j^{-1}\log \underline{q}^j+C_9\\&\leq&d_0^{-1}\log \underline{q}^{-j}+d_0^{-1}\log N_1+\xi_j^{-1}\log \underline{q}^j+C_9\\&=&(d_0^{-1}-\xi_j^{-1})\log \underline{q}^{-j}+C_9+d_0^{-1}\log N_1.
\end{eqnarray*}
So by Lemma \ref{g8}, we get $d_0^{-1}\log M_j+\hat{e}_{M_j}(\mu)<\infty$. Thus by Lemmas \ref{g9}, \ref{interim}, we conclude that $\overline{Q}_0^{d_0}(\mu)<\infty$. This completes the proof of the theorem.


\begin{thebibliography}{10}

\bibitem{Bar:88} M.F. Barnsley,
 Fractals everywhere. Academic Press, New York, London, 1988

\bibitem{Fal:97} K.J. Falconer,
Techniques in fractal geometry, John Wiley \& Sons, 1997.

\bibitem{Gr:95}S. Graf, On Bandt's tangential distribution for self-similar measures.
Monatsh. Math. 120 (1995) 223-246.

\bibitem{GL:00} S. Graf and H. Luschgy,
Foundations of quantization for probability distributons, in: Lecture Notes in Math., vol. 1730, Springer, Berlin, 2000.


\bibitem{GL:01}S. Graf and H. Luschgy, Asymptotics of the quantization error for
self-similar probabilities, Real. Anal. Exchange 26 (2001) 795-810.

\bibitem{GL:04}S. Graf and H. Luschgy,
Quantization for probabilitiy measures with respect to the geometric
mean error, Math. Proc. Camb. Phil. Soc. 136 (2004) 687-717.

\bibitem{GN:98}R. Gray and D. Neuhoff,
Quantization, IEEE Trans. Inform. Theory 44 (1998) 2325-2383.

\bibitem{Hut:81}J.E. Hutchinson,
Fractals and self-similarity Indiana Univ. Math. J. 30 (1981)
713-747.

\bibitem{Kr:08} W. Kreitmeier, Optimal quantization for dyadic
homogeneous Cantor distributions. Math. Nachr. 281 (2008), 1307-1327

\bibitem{Las:06}A. Lasota,
A variational principle for fractal dimensions. Nonlinear Anal.
64 (2006) 618-628

\bibitem{Olsen:08} L. Olsen and N. Snigireva, Multifractal spectra of in-homogenous self-similar measures.
Indiana Univ. Math. J. 57 (2008) 1887-1841

\bibitem{PK:01}K. P\"{o}tzelberger, The quantization dimension of distributions,
Math. Proc. Camb. Phil. Soc. 131 (2001) 507-519.

\bibitem{Schief:94}A. Schief, Separation properties for self-similar sets. Proc. Amer. Math. Soc. 122 (1994), 111-115.

\bibitem{zhu:12} S. Zhu,
A note on the quantization for probability measres with respect to the geometric mean error,
Monatsh. Math. 167 (2012), 295-304.

\bibitem{zhu:13}S. Zhu, Asymptotics of the quantization errors for in-homogeneous self-similar measures supported on self-similar sets, arXiv:1407.3096.

\bibitem{zhu:14} S. Zhu, The quantization for in-homogeneous self-similar measures with in-homogeneous open set condition, arXiv:1407.2212.




\end{thebibliography}
\end{document}